\documentclass[11pt,reqno]{amsart}
\usepackage{amssymb,amsfonts, amsthm, amsmath}
\usepackage{etaremune}
\usepackage{enumerate}
\usepackage{xcolor}
\usepackage[utf8]{inputenc}
\usepackage{color}
\usepackage{comment}
\usepackage{hyperref}
\usepackage{mathtools}
\usepackage[margin=1.25in]{geometry}
\usepackage[normalem]{ulem}
\usepackage{mathrsfs}
\usepackage{todonotes}

 \usepackage{xcolor}
\usepackage{physics}
\usepackage{amsmath}
\usepackage{tikz}
\usepackage{mathdots}
\usepackage{yhmath}
\usepackage{cancel}
\usepackage{color}
\usepackage{siunitx}
\usepackage{array}
\usepackage{multirow}
\usepackage{amssymb}
\usepackage{gensymb}
\usepackage{tabularx}
\usepackage{extarrows}
\usepackage{booktabs}
\usetikzlibrary{fadings}
\usetikzlibrary{patterns}
\usetikzlibrary{shadows.blur}
\usetikzlibrary{shapes}

\newtheorem{theorem}{Theorem}[section]
\newtheorem*{theorem*}{Theorem}
\newtheorem{corollary}[theorem]{Corollary}

\newtheorem{prop}[theorem]{Proposition}
\newtheorem{lemma}[theorem]{Lemma}
\theoremstyle{definition}

\newtheorem{remark}[theorem]{Remark}

\newtheorem{claim}{Claim}[section]

\newcommand\pd[2]{\frac{\partial #1}{\partial #2}}

\newcommand{\ft}{\mathfrak{t}}

\newcommand\Rb[1]{\mathbb{R}^{#1}}

\numberwithin{equation}{section}
\numberwithin{figure}{section}

\newcommand{\cC}{\mathcal{C}}
\newcommand{\cD}{\mathcal{D}}
\newcommand{\cH}{\mathcal{H}}

\newcommand{\cE}{\mathcal{E}}

\newcommand{\bx}{\mathbf{x}}
\newcommand{\by}{\mathbf{y}}
\newcommand{\bF}{\mathbf{F}}
\newcommand{\cF}{\mathcal{F}}
\newcommand{\bOh}{\mathbf{0}}
\newcommand{\bH}{\mathbf{H}}
\DeclareMathOperator{\Gap}{Gap}
\DeclareMathOperator{\Graph}{graph}
\DeclareMathOperator{\tra}{tr}
\newcommand{\cM}{\mathcal{M}}

\newcommand{\RR}{\mathbb{R}}
\newcommand{\NN}{\mathbb{N}}

\DeclareMathOperator{\supp}{supp}
\DeclareMathOperator{\sing}{sing}
\DeclareMathOperator{\Id}{Id} 
\DeclareMathOperator{\topint}{int}

\title[MCF from conical singularities]{Mean curvature flow from conical singularities}
\author{Otis Chodosh}
\address{Department of Mathematics, Bldg.\ 380, Stanford University, Stanford, CA 94305, USA}
\email{ochodosh@stanford.edu}
\author{J. M. Daniels-Holgate}
\address{Einstein Institute of Mathematics, Edmond J. Safra Campus, The Hebrew University of Jerusalem, Givat Ram. Jerusalem, 9190401, Israel} 
\email{joshua.daniels-holgate@mail.huji.ac.il} 
\author{Felix Schulze}
\address{Department of Mathematics, Zeeman Building, University of Warwick, Gibbet Hill Road, Coventry CV4 7AL,
UK}

\email{felix.schulze@warwick.ac.uk} 

\begin{document}

\maketitle
\begin{abstract}
 We prove Ilmanen's resolution of point singularities conjecture by establishing short-time smoothness of the level set flow of a smooth hypersurface with isolated conical singularities. This shows how the mean curvature flow evolves through asymptotically conical singularities. 

Precisely, we prove that the level set flow of a smooth hypersurface $M^n\subset \mathbb{R}^{n+1}$, $2\leq n\leq 6$, with an isolated conical singularity is modeled on the level set flow of the cone. In particular, the flow fattens (instantaneously) if and only if the level set flow of the cone fattens. 
\end{abstract}

\section{Introduction}

 A family of smooth hypersurfaces $M(t)$ is a mean curvature flow if 
\[
(\tfrac{\partial}{\partial t}\bx)^\perp = \bH_{M(t)}(\bx),
\]
where $\bH_{M(t)}(\bx)$ is the mean curvature vector of $M(t)$ at $\bx$. Mean curvature flow is the gradient flow of area. We recall that the mean curvature flow, $M(t)$, from a smooth, compact hypersurface $M(0)\subset \RR^{n+1}$  is guaranteed to become singular in finite time, moreover, well-posedness and regularity of the flow can break down after the onset of certain singularities (cf.\ \cite{White:ICM}). 

In the present article, we quantify the short-time regularity and well-posedness of the level set flow from a smooth compact hypersurface with an isolated singularity\footnote{For simplicity of notation we only consider a single singularity, but everything here would generalize easily to the case of finitely many isolated singularities, each modeled on a smooth cone.} modeled on any smooth cone $\cC$. Recalling 
\cite{ChoSch}, such hypersurfaces can appear as the singular time-slice of a flow encountering a singularity modeled on an asymptotically conical self-shrinker. Our results hence demonstrate how one can flow through such a singularity.

Before stating our results, we recall that the level set flow (cf.~\cite{OS,CGG,EvansSpruck1,ilmanen}) of a closed set $X$ is the unique maximal assignment of closed sets $t\mapsto F_t(X)$ with $F_0(X) = X$, such that $F_t(X)$ avoids smooth flows (see Section \ref{subsec:level-set-flow}). If the level set flow $F_t(X)$ develops an interior at $t=T$, we say that the flow \emph{fattens} at time $T$.


Our main results can be stated as follows:
\begin{theorem}[Fattening dichotomy]\label{theo:fatt-inf}
For $2\leq n \leq 6$, suppose that $M^n\subset \RR^{n+1}$ is a smooth hypersurface with an isolated conical singularity modeled on a smooth cone $\cC$. Then the level set flow of $M$ fattens instantly if and only if the level-set flow from the cone $\cC$ fattens. 
\end{theorem}

Fattening of $\cC$ implies fattening of $M$ is proven in Theorem \ref{theo:fattening}, whilst non-fattening of $\cC$ implies short-time non-fattening of $M$ can be found in Corollary \ref{coro:non-fat}.

A fortiori, Theorem \ref{theo:fatt-inf} is a consequence of the following results (more precisely Theorem \ref{theo:structure-theo} and Theorem \ref{theo:uniqueness}), which give a precise description of the level-set flow near the conical singularity of $M$.
\begin{theorem}[Structure theorem for the level set flow]\label{theo:structure-theo}
For $2\leq n \leq 6$, suppose that $M^n\subset \RR^{n+1}$ is a smooth hypersurface with an isolated conical singularity modeled on a smooth cone $\cC$ at $\bOh$. Then, there is a $T>0$ such that the outermost mean curvature flows of $M$ are smooth for $t \in (0,T)$. Moreover, $t^{-1/2}F_t(M)$ converges in the local Hausdorff sense to $F_1(\cC)$ as $t\searrow 0$.
\end{theorem}

We provide a refinement of this statement below, which, in aggregate with the aforementioned work \cite{ChoSch}, can be considered as a canonical neighbourhood theorem for asymptotically conical singularities. Before stating this result, we provide a brief exposition of the Hershkovits--White framework applicable to the present context. (See Section \ref{sec:outermost} for a rigorous discussion.)


 \begin{figure}[h]
    \centering

\includegraphics{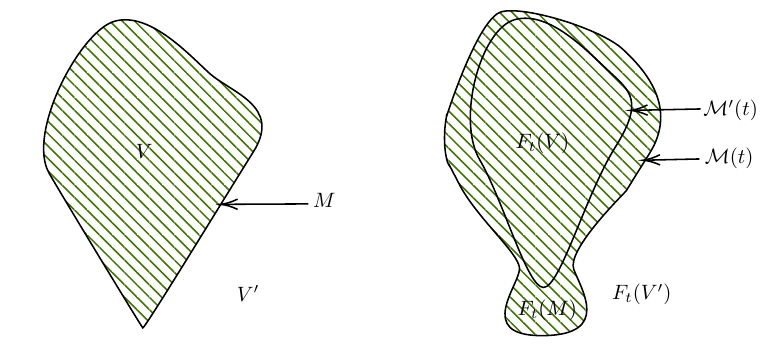}
    \caption{Left: A hypersurface $M$ with isolated conical singularity, with interior $V$ and exterior $V'$. Right: The level set flow of each at time $t$.}
    \label{fig:compact}
\end{figure}

We first consider the compact case, illustrated in Figure \ref{fig:compact}. Recall that the outer flow, $\cM$ is the space-time boundary of the level set flow $F_t(V)$ of the interior $V$ of $M$. Similarly the inner flow $\cM'$ is the space-time boundary of the level set flow $F_t(V')$ of the exterior $V'$ of $M$. Turning to the cone $\mathcal{C}$ as illustrated in Figure \ref{fig:cone}, we note dilation invariance and uniqueness of the level set flow yields $F_t(\cC)=\sqrt{t}F_1(\cC)$. Denote $W$ and $W'$ the interior and exterior of the cone $\cC$ and assume we have chosen these conistently with the interior and exterior of $M$. Let $\Sigma':= \partial F_1(W')$ and $\Sigma:= \partial F_1(W)$ and observe that $\partial F_t(W) = \sqrt{t}\Sigma$, $\partial F_t(W') = \sqrt{t}\Sigma'$. Note, when $2\leq n\leq 6$, $\Sigma,\Sigma'$ will be smooth (this is the source of the dimension restriction above; for $n>6$, the above theorems continue to hold if we impose the additional condition that the outermost expanders for $\cC$ are smooth). 

In the sequel we will refer to $\cM(t),\cM'(t)$ as \emph{outermost} flows and $\Sigma$, $\Sigma'$ as the outermost expanders. The next result shows that the outermost expanders approximate the outermost flows. 

\begin{theorem}[Canonical neighbourhood theorem for outermost flows]\label{theo:exist-inf}
For $2\leq n \leq 6$, suppose that $M^n\subset \RR^{n+1}$ is a smooth hypersurface with an isolated conical singularity at $\bOh$. Assume the conical singularity is modeled on a smooth cone $\cC$ with outermost expanders $\Sigma,\Sigma'$ labeled as above. Then, $t^{-1/2}\mathcal{M}(t)$ (resp.\ $t^{-1/2}\mathcal{M}'(t)$) converges to $\Sigma$ (resp.\ $\Sigma'$) locally smoothly as $t\searrow 0 $.\end{theorem}
Theorem \ref{theo:exist-inf} resolves the ``resolution of point singularities'' conjecture of Ilmanen \cite[Problem 16]{ilmanen2003problems}. Smoothness can be found in Corollary \ref{coro:smoothness} and the forward blow-up statement (including the convergence of the outermost Brakke flows to the outermost expanders) can be found in Theorem \ref{theo:fattening}.



\begin{figure}[t]
    \centering
\includegraphics{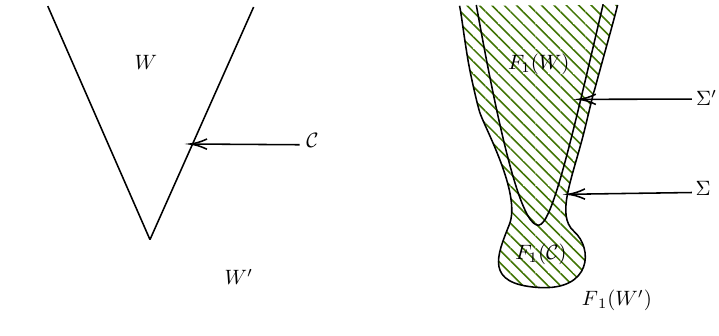}

\caption{Left: The initial cone $\mathcal{C}$ with interior $W$ and exterior $W'$. Right: The level set flow of each region at time $t=1$. }\label{fig:cone}
\end{figure}

Note that if $\cC$ does not fatten then $\Sigma = \Sigma'$ and Theorem \ref{theo:structure-theo} trivially holds. In particular, any flows starting from $M$ are smooth for a short time and modeled on the unique expander asymptotic to $\cC$. This implies that two such flows separate like $o(t^{1/2})$, but this does not a priori imply there is only one such flow. This is the content of our final result.

\begin{theorem}[Uniqueness]\label{theo:uniqueness}
For $2\leq n \leq 6$, suppose that $M^n\subset \RR^{n+1}$ is a smooth hypersurface with an isolated conical singularity at $\bOh$, modeled on a smooth cone $\cC$ which does not fatten. Then, there is $T>0$ such that the outermost flows of $M$ agree (and are smooth) for $t\in (0,T)$. Especially, the evolution of $M$ is unique on this time interval. 
\end{theorem}

Smoothness follows again from Corollary \ref{coro:smoothness} and for uniqueness see Corollary \ref{coro:non-fat}.

\begin{remark}\,
As a consequence of the work of Brendle \cite{Brendle:genus0} any asymptotically conical shrinker must have non-zero genus and by work of Ilmanen--White \cite[p.\ 21]{Trieste} the inner and outer expanders are topological planes. Combined with Chodosh--Schulze \cite[Corollary 1.2]{ChoSch}, the results presented in this work demonstrate strict genus drop through any isolated conical singularities that form in a multiplicity one flow. We note that the full ``strict genus monotonicity conjecture'' of Ilmanen \cite[Problem 13]{ilmanen2003problems} at non-generic singularities (for outermost flows) was recently resolved by Bamler--Kleiner \cite{BamlerKleiner23} by combining their resolution of Ilmanen’s multiplicity one conjecture with the strict genus drop results for one-sided perturbations of Chodosh--Choi--Schulze--Mantoulidis \cite{CCMS:generic1, ccs23}.
\end{remark}

\subsection{Related work} 

The study of fattening and non-fattening of conical singularities has received considerable attention. In particular, in their first work on the level set flow, Evans--Spruck already observed \cite[\S 8.2]{EvansSpruck1} that the cone $\mathcal{C}:=\{xy=0\}\subset \RR^2$ and a figure eight will fatten. Note that a figure eight is a smooth curve in $\RR^2$ with an isolated conical singularity modeled on the cone $\mathcal{C}$ in the terminology of this paper (and our results would apply without change to this setting). Fattening has been subsequently studied by many authors, see \cite{SonSou,DG:new,White:questDG,ilmanen,Trieste,AltAngGig,AngChopIlm,NP:numerics,angenent2002fattening,White:ICM,Helm, DingExp} for a non-exhaustive list. 

More recently, Hershkovits--White \cite{hershwhite} introduced a powerful framework for analysing the level set flow, which they applied to show non-fattening through mean-convex singularities. Combining their work with the resolution of the mean-convex neighborhood conjecture by Choi--Haslhofer--Herskovits \cite{chh18} (cf.\ \cite{ChoiHaslhoferHershkovitsWhite}), it follows that fattening does not occur if all singularities are either round cylinders of the form $\mathbb{S}^{n-1}\times \RR$ or round spheres $\mathbb{S}^{n}$. We also draw attention to the recent studies of asymptotically conical expanders by Deruelle--Schulze \cite{DS20} and Bernstein--Wang \cite{BW:space,BW:compact,BW:mount,BW19,BW:top-unique-expand,BW18}. In particular, Bernstein--Wang have used these results to prove a low-entropy Schoenflies theorem \cite{BW:schoen} (cf.\ \cite{CCMS:generic1,CCMS:low-ent-gen,Daniels-Holgate}) and have announced applications to the study of low-entropy cones. See also the work of Chen \cite{Chen:rot1,Chen:rot2,chen:existunique}. 

Finally, we note that the question of evolving a Ricci flow through a singularity modelled on the evolution of an asymptotically conical gradient shrinking soliton is also of considerable interest (but we note that the analogues of Theorem \ref{theo:fatt-inf} and \ref{theo:exist-inf} and the resolution of point singularities are not understood in general). In particular, expanders have been studied in \cite{ShulzeSimon,Der16,DS:rel,BamChen} and flows have been constructed out of initial Riemannian manifolds with isolated conical singularities modeled on non-negatively curved cones over spheres \cite{GS}. Moreover, ``fattening'' of the cone at infinity of a shrinking gradient Ricci soliton has been constructed in \cite{AK}.


\subsection{Strategy of proof}

Optimistically, one might hope that the resolution of a conical singularity is always modeled on expanders, just as tangent flows are always modeled on self-shrinkers. Indeed, one might expect a forward monotonicity formula would control the forward blow-ups (the (subsequential) weak limits of $\lambda^2 M(\lambda^{-1}t)$ as $\lambda\to\infty$) but there appear to be serious issues to make this rigorous in the setting of isolated conical singularities (cf.\ \cite[p.\ 25]{Trieste}). We do note that in the setting of flows coming out of cones, Bernstein--Wang have obtained a version of forwards monotonicity \cite{BW19} (generalizing to the dynamical setting the relative expander entropy of Deruelle--Schulze \cite{DS20}) and Chen \cite{chen:existunique} has constructed non-self-expanding flows from cones. However, it remains unclear if/how monotonicity based methods could prove that forward blowups of outermost flows are outermost expanders (or even that they are smooth). 

In this article we take a completely different approach (avoiding forwards monotonicity entirely). Instead, we find barriers that push the outermost flows onto the outermost expanders in the forward blowup limit. A closely related construction proves uniqueness of two flows with the same outermost expander blowup limit. The construction of these barriers combines two key spectral properties of an outermost expander $\Sigma$:
\begin{enumerate}
\item The outermost expander minimizes weighted area to the outside, so the linearized expander operator (cf.\ \eqref{eq:definiton-Jacobi}) is non-negative $L_\Sigma\geq 0$. In particular, there is a positive eigenfunction $\phi_{3R}$ on $\Sigma \cap B_{3R}(\bOh)$ with positive eigenvalue $\mu_{3R}>0$. 
\item The outermost expander is the one-sided limit of expanders asymptotic to nearby cones, which yields a positive Jacobi field $L_\Sigma v = 0$ with $v$ growing linearly at infinity. 
\end{enumerate}
The ``interior'' barrier is then formed by taking the graph over $\Sigma$ of a small multiple of $f : = v+\alpha \phi_{3R}$. Because $L_\Sigma f = -\alpha \mu \phi_{3R}$ this can be seen to be a strict barrier in $B_{2R}(\bOh)$, pushing (rescaled) mean curvature flows towards $\Sigma$.

To prove that the flow fattens if the cone fattens (Theorem \ref{theo:fattening}), we can weld (in the sense of  Meeks--Yau \cite{MeeksYau}) the graph of $f$ to the graph of a constant function $h$ over $\Sigma$ to obtain a global barrier $\Gamma_s$ over $\Sigma$ (note that $L_\Sigma h = (|A_\Sigma|^2 - \frac 12)h$ is $<0$ outside of a sufficiently large compact set). (See Proposition \ref{prop:fattening-barrier-final}.) Now, the key observation is that the forward blowups of the outermost flow will lie below $\Gamma_s$ outside of a sufficently large set, since the forward blowups must lie in the level set flow of the cone (which decay towards the cone) while $\Gamma_s$ has height $\sim sh$ over the cone near infinity (see Claim \ref{claim:boundary-true-1}). 

In particular, this proves that the outermost flows have forward blowup at $\bOh$ equal to the outermost flows of the cone $\cC$. To prove that the flow does not fatten if the cone does not fatten, it thus suffices to consider two flows $\cM^1(t),\cM^2(t)$ which have forwards blowup given by the same outermost expander. We construct a barrier in this situation by welding (in the sense of  Meeks--Yau \cite{MeeksYau}) the graph of $\pm sf$ (denoted $\Gamma_{\textrm{close},s}^\pm(t)$) to the normal graph of $s h \sqrt{t}$ over $\cM^1(t)\setminus B_{\sqrt{t}R}(\bOh)$ (denoted $\Gamma_{\textrm{far},s}^\pm(t)$). The barriers then pinch $\cM^2$ towards $\cM^1$ from above and below as $s\to0$ proving uniqueness. This can be seen in Figure \ref{fig:unique}.

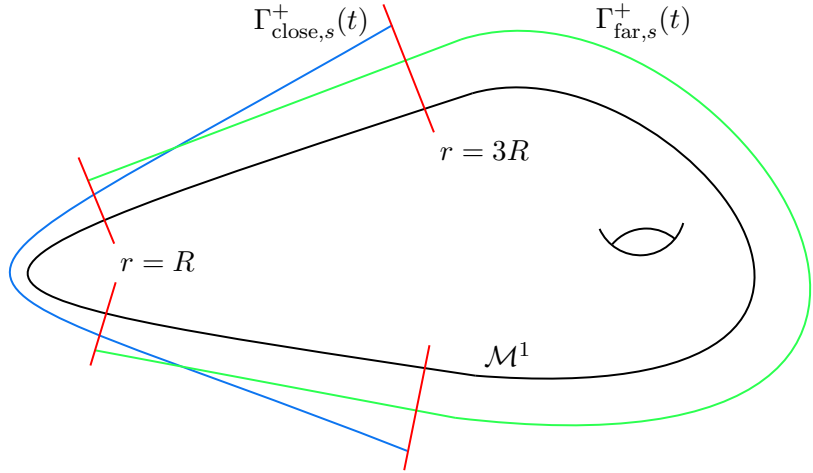
\begin{figure}
\tikzset{every picture/.style={line width=0.75pt}} 

\begin{tikzpicture}[x=0.75pt,y=0.75pt,yscale=-0.75,xscale=0.75]

\draw    (441,260) .. controls (40,200) and (40,200) .. (439,70) ;
\draw [color={rgb, 255:red, 13; green, 117; blue, 240 }  ,draw opacity=1 ]   (395.2,310.62) .. controls (46.2,173.62) and (35.2,226.15) .. (385.2,24.17) ;
\draw [color={rgb, 255:red, 43; green, 255; blue, 80 }  ,draw opacity=1 ]   (180.39,128.83) -- (431.39,33.76) ;
\draw [color={rgb, 255:red, 43; green, 255; blue, 80 }  ,draw opacity=1 ]   (185.45,242.96) -- (427.45,288.36) ;
\draw [color={rgb, 255:red, 255; green, 0; blue, 0 }  ,draw opacity=1 ]   (174.2,113.17) -- (198.2,171.17) ;
\draw [color={rgb, 255:red, 255; green, 0; blue, 0 }  ,draw opacity=1 ]   (199.2,197.17) -- (182.2,253.17) ;
\draw [color={rgb, 255:red, 255; green, 0; blue, 0 }  ,draw opacity=1 ]   (379.2,10.17) -- (413.2,96.23) ;
\draw [color={rgb, 255:red, 255; green, 0; blue, 0 }  ,draw opacity=1 ]   (410.2,239.53) -- (393.2,323.62) ;
\draw    (439,70) .. controls (574.2,31.23) and (787.2,287) .. (441,260) ;
\draw  [draw opacity=0] (580.85,157.26) .. controls (577.48,169.04) and (567.04,178.01) .. (554.11,178.93) .. controls (541.18,179.84) and (529.58,172.41) .. (524.59,161.22) -- (552,149) -- cycle ; \draw   (580.85,157.26) .. controls (577.48,169.04) and (567.04,178.01) .. (554.11,178.93) .. controls (541.18,179.84) and (529.58,172.41) .. (524.59,161.22) ;  
\draw  [draw opacity=0] (533.15,171.56) .. controls (538.7,165.04) and (546.98,160.94) .. (556.21,161) .. controls (563.55,161.05) and (570.26,163.74) .. (575.44,168.15) -- (556,191) -- cycle ; \draw   (533.15,171.56) .. controls (538.7,165.04) and (546.98,160.94) .. (556.21,161) .. controls (563.55,161.05) and (570.26,163.74) .. (575.44,168.15) ;  
\draw [color={rgb, 255:red, 43; green, 255; blue, 80 }  ,draw opacity=1 ]   (431.39,33.76) .. controls (606.2,-22) and (860.2,340) .. (427.45,288.36) ;

\draw (520,7.4) node [anchor=north west][inner sep=0.75pt]    {$\Gamma_{\textrm{far},s}^+(t)$};
\draw (445,235.4) node [anchor=north west][inner sep=0.75pt]    {$\cM^{1}$};
\draw (200.2,174.57) node [anchor=north west][inner sep=0.75pt]    {$r=R$};
\draw (415.2,99.63) node [anchor=north west][inner sep=0.75pt]    {$r=3R$};
\draw (290,8) node [anchor=north west][inner sep=0.75pt]    {$\Gamma _{\textrm{close},s}^{+}(t)$};

\end{tikzpicture}

\caption{The barrier construction to prove uniqueness of flows with the same (outermost) expander $\Sigma$ as forward blowup at $\bOh$. }\label{fig:unique}
\end{figure}

\subsection{Organization}

In Section \ref{sec:prelim} we collect some preliminary definitions and facts to be used later. In Section \ref{sec:expander-barriers} we construct barriers graphical over the expander and then use these barriers to prove that the level set flow is locally modeled on the level set flow of the cone in Section \ref{fattening}. In Section \ref{sec:glob-barrier} we construct global barriers over a flow that's modeled on an outermost expander near the conical singularity and then use these to prove uniqueness of such flows in Section \ref{sec:unique}. Finally, we collect some results about graphs over expanders in Appendix \ref{sec:graphs}. 

\subsection{Acknowledgements} 

O.C. was supported by a Terman Fellowship and an NSF grant (DMS-2304432). J.M.D-H. was supported by the Warwick Mathematics Institute Centre for Doctoral Training. This project has received funding from the European Research Council (ERC) under the European Union’s Horizon 2020 research and innovation programme, grant agreement No 101116390.

\section{Preliminaries}\label{sec:prelim}
In this section we collect some preliminary definitions, conventions, and results. 

\subsection{Spacetime and the level set flow}\label{subsec:level-set-flow}
We define the \emph{time} map $\ft : \RR^{n+1}\times \RR\to \RR$ to be the projection $\ft(\bx,t) = t$. For $E\subset \RR^{n+1}\times \RR$ we will write $E(t) : = E\cap\ft^{-1}(t)$. The knowledge of $E(t)$ for all $t$ is the same thing as knowing $E$, so we will often ignore the distinction. 

For a compact $n$-manifold $M$ (possibly with boundary), we consider $f: M\times [a,b] \to \RR^{n+1}$ so that (i) $f$ is continuous (ii) $f$ is smooth on $(M\setminus\partial M)\times (a,b]$ (iii) $f|_{M\times \{t\}}$ is injective for each $t \in [a,b]$ and (iii) $t\mapsto f(M\setminus\partial M,t)$ is flowing by mean curvature flow. In this case we call
\[
\cM : =\cup_{t\in [a,b]} f(M,t)\times \{t\} \subset \RR^{n+1}\times \RR
\]
a \emph{classical mean curvature flow} and define the \emph{heat boundary} of $\cM$ by
\[
\partial\cM : = f(M,a) \cup f(\partial M,[a,b]). 
\]
Classical flows that intersect must intersect in a point that belongs to at least one of their heat boundaries (cf.\ \cite[Lemma 3.1]{White:topology-weak}). 

For $\Gamma \subset \RR^{n+1}\times [0,\infty)$, $\cM \subset \RR^{n+1}\times \RR$ is a \emph{weak set flow} (generated by $\Gamma$) if $\cM(0) =\Gamma(0)$ and if $\cM'$ is a classical flow with $\partial \cM'$ disjoint from $\cM$ and $\cM'$ disjoint from $\Gamma$ then $\cM'$ is disjoint from $\cM$. There may be more than one weak set flow generated by $\Gamma$. 

The biggest such flow is called the \emph{level set flow}, which can be constructed as follows: For $\Gamma \subset \RR^{n+1}\times [0,\infty)$ as above, we set
\[
W_0 : = \{(\bx,0) : (\bx,0) \not \in \Gamma\}
\]
and then let $W_{k+1}$ denote the union of all classical flows $\cM'$ with $\cM'$ disjoint from $\Gamma$ and $\partial\cM'\subset W_k$. The \emph{level set flow} generated by $\Gamma$ is then defined by
\[
\cM : = (\RR^{n+1}\times [0,\infty))\setminus \cup_k W_k \subset \RR^{n+1}\times [0,\infty). 
\]
See \cite{EvansSpruck1,ilmanen,White:topology-weak}. If $\Gamma\subset \RR^{n+1}\times \{0\}$, we will write $F_t(\Gamma) : = \cM(t)$ for the time $t$ slice of the corresponding level set flow. 

Fix $\Gamma \subset \RR^{n+1}$ closed. We say that the level set flow of $\Gamma$ is \emph{non-fattening} if $F_t(\Gamma)$ has no interior for each $t \geq 0$. This condition holds generically for compact $\Gamma\subset \RR^{n+1}$, namely if $u_0$ is a continuous function with compact level sets $u_0^{-1}(s)$ then the level set flow of $u_0^{-1}(s)$ fattens for at most countably many values of $s$, see \cite[\S 11.3-4]{ilmanen}.

\subsection{Integral Brakke flows} An ($n$-dimensional\footnote{Of course one can consider $k$-dimensional flows in $\RR^{n+1}$ but we will never do so in this paper, so we will often omit the ``$n$-dimensionality'' and implicitly assume that all Brakke flows are flows of ``hypersurfaces.''}) integral Brakke flow in $\RR^{n+1}$ is a $1$-parameter family of Radon measures $(\mu(t))_{t\in I}$ so that
\begin{enumerate}
\item For almost every $t\in I$ there is an integral $n$-dimensional varifold $V(t)$ with $\mu(t) = \mu_{V(t)}$ and so that $V(t)$ has locally bounded first variation and mean curvature $\bH$ orthogonal to $\textrm{Tan}(V(t),\cdot)$ almost everywhere. 
\item For every bounded interval $[t_1,t_2]\subset I$ and $K\subset \RR^{n+1}$ compact, we have
\[
\int_{t_1}^{t_2} \int_K (1+|\bH|^2) d\mu(t) dt < \infty.
\]
\item If $[t_1,t_2]\subset I$ and $f\in C^1_c(\RR^{n+1}\times [t_1,t_2])$ has $f\geq 0$ then
\[
\int f(\cdot,t_2) d\mu(t_2) - \int f(\cdot,t_1) d\mu(t_1) \leq \int_{t_1}^{t_2} \int (-|\bH|^2 f + \bH \cdot\nabla f + \tfrac{\partial f}{\partial t}) d\mu(t) dt.
\]
\end{enumerate}
We will sometimes write $\cM$ to represent a Brakke flow. 

We define the support of $\cM = (\mu(t))_t$ to be $\overline{\cup_t \supp \mu(t) \times \{t\}} \subset \RR^{n+1}\times \RR$.  It is useful to recall that the support a Brakke flow (with $t \in [0,\infty)$) is a weak set flow (generated by $\supp \mu(0)$) \cite[10.5]{ilmanen}. 

We say that a sequence of integral Brakke flows $\cM_i$ converges to another integral Brakke flow $\cM$ (written $\cM_i \rightharpoonup \cM$) if $\mu_i(t)$ weakly converges to $\mu(t)$ for all $t$ and for almost every $t$, after passing to a further subsequence depending on $t$, the associated integral varifolds converge $V_i(t) \to V(t)$. (Recall that if $\cM_i$ is a sequence of integral Brakke flows with uniform local mass bounds then a subsequence converges to an integral Brakke flow \cite[\S 7]{ilmanen}.)

For a Brakke flow $\cM$ and $\lambda>0$ we write $\cD_\lambda(\cM)$ for the ``dilated'' Brakke flows with measures satisfying $\mu_\lambda(t)(A) = \lambda^n \mu(\lambda^{-2}t)(\lambda^{-1}A)$. 


\subsection{The inner/outer flows of a level set flow}\label{sec:outermost}

We collect results of \cite{hershwhite} on weak set flows and outermost flows and show that they are also applicable (with minor modifications) to the flow of more general initial data. 

\begin{prop}[{\cite[Proposition A.3]{hershwhite}}]\label{thm:app.1}
Suppose that $F$ is a closed subset of $\RR^{n+1}$, and let $\cM \subset \RR^{n+1} \times \RR^+$ be its level set flow. Set:
$$ M(t) := \{\mathbf{x} \in \RR^{n+1} : (\mathbf{x},t) \in \partial \cM \}\, .$$
Then $t \mapsto M(t)$ is a weak set flow.
\end{prop}

In what follows, we assume that $F$ is the closure of its interior in $\RR^{n+1}$ (we will call\footnote{Note that this slightly extends  the definition in \cite{hershwhite}, where $\partial F$ ($\partial U$ in their notation) would be a compact, smooth hypersurface. This extension allows us to flow from non-compact and non-smooth initial surfaces. This does not change anything in the analysis of \cite{hershwhite}.}  such a set $F$  \emph{admissible}).
Let $F':= \overline{F^c}$, denote the level set flows of $F$, $F'$ by $\cM$, $\cM'$, and set $F(t):= \cM(t)$, $F'(t):= \cM'(t)$. In line with Proposition \ref{thm:app.1}, we set:
\begin{align*}
M(t) &:= \{ (\mathbf{x},t) \subset \RR^{n+1} : \mathbf{x} \in \partial \cM\},\\
M'(t) &:= \{ (\mathbf{x},t) \subset \RR^{n+1} : \mathbf{x} \in \partial \cM'\}.
\end{align*}
(Here $\partial \cM$, $\partial \cM'$ are the relative boundaries of $\cM$, $\cM'$ as subsets of $\RR^{n+1}\times \RR^+$). We call
\[ t \mapsto M(t), \; t \mapsto M'(t) \]
the \emph{outer} and \emph{inner} flows of $M:= \partial F$. By Proposition \ref{thm:app.1}, $M(t)$, $M'(t)$ are contained in the level set flow generated by $M$. Furthermore,
\[
M(t) = \lim_{\tau \nearrow t} \partial F(\tau) 
\]
for all $t > 0$, and $M(t) = \partial F(t)$ for all but countably many $t$. See \cite[Theorems B.5,  C.10]{hershwhite}. Note that \cite[Theorems B.5]{hershwhite} directly carries over to $M = \partial F$ where $F$ is admissible. 

We will say that an admissible set $F$ is \emph{smoothable}, if the following holds: There exist compact regions $F\Subset F_i$ with smooth boundaries such that 
\begin{itemize}
\item[(1)] For each $i$, $F_{i+1}$ is contained in the interior of $F_{i}$.
\item[(2)] $\cap F_i = \topint F$.
\item[(3)] $\cH^{n}\lfloor M$ is a Radon measure and $\cH^{n}\lfloor \partial F_i \rightarrow \cH^{n}\lfloor M$.
\item[(4)] There is $\Lambda > 0$ so that for any $p \in \RR^{n+1}$ and $\rho>0$ it holds that $|\partial F_i \cap B_\rho(p)| \leq \Lambda \rho^n$.
\end{itemize}
By perturbing $F_i$ slightly, we can also assume that
\begin{itemize}
\item[(5)] the level set flow of $\partial F_i$ never fattens.
\end{itemize}
Choose integral Brakke flows $t\mapsto \mu_i(t)$ starting from $\mu_i(0) = \cH^n\lfloor \partial F_i$ via elliptic regularization. Assume that $\mu_i$ limits to $t\mapsto \mu(t)$ in the sense of Brakke flows. Note that the flow $t \in [0,\infty) \mapsto \mu(t)$ is an integral, unit-regular Brakke flow with $\mu(0) = \cH^{n} \lfloor M$

We do the same hold with $F'$ replacing $F$ and so on. We then directly generalize \cite[Theorems B.6,  B.8]{hershwhite}.  The proof extends verbatim.

\begin{prop}\label{prop:inner-outermost-flow} Assume $F$ is admissible and smoothable with $M = \partial F$ The Brakke flow $t\mapsto \mu(t)$ has spacetime support equal to the spacetime set swept out by $t \in [0,\infty) \mapsto M(t)$, where $t  \mapsto M(t)$ is the outer flow of $M$. More precisely, for $t>0$, the Gaussian density of the flow $\mu(\cdot)$ at $(\mathbf{x},t)$ is $>0$ if and only if $\mathbf{x} \in M(t)$. The analogous statement holds for the inner flow $t  \mapsto M'(t)$ of $M$.
\end{prop}

\subsection{Density, Huisken's monotonicity, and entropy}\label{subsec:dens-mon-entropy} For $X_0 = (\bx_0,t_0)\in\RR^{n+1}\times \RR$ we consider the ($n$-dimensional) backwards heat kernel based at $X_0$:
\begin{equation}\label{eq:gauss.dens.defn}
\rho_{X_0} (\bx,t) : = (4\pi(t_0-t))^{-\frac n2} \exp\left( -\frac{|\bx-\bx_0|^2}{4(t_0-t)}\right)
\end{equation}
for $\bx\in\RR^{n+1},t<t_0$. For $\cM$ a Brakke flow defined on $[T_0,\infty)$, $t_0>T_0$ and $0<r\leq\sqrt{T_0-t_0}$, we set
\[
\Theta_\cM(X_0,r) : = \int \rho_{X_0}(\bx,t_0-r^2) d\mu(t_0-r^2). 
\]
Huisken's monotonicity formula \cite{Huisken:sing,Ilmanen:singularities} implies that $r\mapsto \Theta_{\cM}(X_0,r)$ is non-decreasing (and constant only for a shrinking self-shrinker centered at $X_0$). In particular we can define the density of $\cM$ at such $X_0$ by
\[
\Theta_\cM(X_0) : = \lim_{r\searrow 0} \Theta_\cM(X_0,r). 
\]
We call an integral Brakke flow $\cM$ \emph{unit-regular} if $\cM$ is smooth in a forwards-backwards space-time neighborhood of any space-time point $X$ with $\Theta_\cM(X) = 1$. Note that we can then write $\sing\cM = \{X \in\RR^{n+1}\times \RR : \Theta_\cM(X) > 1\}$.  Note that by \cite[Theorem 4.2]{SchulzeWhite} the class of unit-regular integral Brakke flows is closed under the convergence of Brakke flows. Furthermore, combining  \cite[Lemma 4.1]{SchulzeWhite} and \cite{White:Brakke} it follows that there is $\varepsilon_0>0$, depending only on dimension, such that every point $X\in \sing\cM$ has $\Theta_\cM(X) \geq 1 + \varepsilon_0$. Upper semi-continuity of density then implies that $\sing\cM$ is closed.



\subsection{Cones and Self-Expanders}\label{subsec:cones}
Consider $S \subset \mathbb{S}^n$ a smooth, embedded, closed hypersurface. We then call the cone over $S$, denoted by $\mathcal{C}=\mathcal{C}(S) \subset \RR^{n+1}$, \emph{smooth}. We say that $M\subset\Rb{n+1}$ is a \emph{smooth hypersurface with a conical singularity at $\bx_0$ modelled on the cone $\cC$} if:
    \begin{enumerate}
        \item $M\backslash\{\bx_0\}$ is a smooth (embedded) hypersurface,
        \item $\lim_{\rho\to\infty}\rho\cdot (M-\bx_0) = \cC$, 
    \end{enumerate}
    where the convergence is in $C^\infty_{loc}(\Rb{n+1}\backslash \{0\})$. Note that a hypersurface with conical singularities is admissible and smoothable in the sense of Section \ref{sec:outermost}, see also \cite[Appendix E]{CCMS:generic1}. 
    
    Similarly, we say that a hypersurface $M\subset\Rb{n+1}$ is  \emph{(smoothly) asymptotic} to $\cC$ if     
    \begin{align*}
        \lim_{\rho\searrow0}\rho \cdot M=\mathcal{C}
    \end{align*}
    in $C^\infty_{loc}(\Rb{n+1}\backslash\{0\})$.

A natural class of solutions to mean curvature flow, starting from an initial (smooth) cone $\cC$, are \emph{self-similarly expanding solutions}, i.e.~solutions given by
\begin{equation}\label{eq:expanding_solutions}
 t\mapsto \sqrt{t} \cdot \Sigma
 \end{equation}
for $t>0$, where $\Sigma$ is asymptotic to $\cC$. These solutions are invariant under parabolic rescalings \emph{forward in time}. The condition that \eqref{eq:expanding_solutions} is a mean curvature flow yields an elliptic equation for $\Sigma$, given by  
\begin{align}\label{eq:expander}
    \mathbf{H}_\Sigma(\mathbf{x})-\frac{\mathbf{x}^\perp}{2} = 0.
\end{align}
We call $\Sigma$ a \emph{self-expander} and denote the corresponding immortal solution to mean curvature flow by $\cM_\Sigma$. Alternatively, self-expanders are critical points (under compact perturbations) of the expander functional 
$$ \cE(M) = \int_Me^\frac{|\bx|^2}{4}\, d\cH^n\, .$$
We call a self-expander $\Sigma$ stable if the second variation of $\cE$ is non-negative under compact perturbations, i.e.~if
$$  \int_\Sigma \varphi (-L_\Sigma\varphi) \, e^\frac{|\bx|^2}{4}\, d\cH^n \geq 0 $$
for all $\varphi \in C^\infty_c(\Sigma)$, where $L_\Sigma$ is the corresponding Jacobi operator given by
 \begin{equation}\label{eq:definiton-Jacobi}
        L_{\Sigma}=\Delta_\Sigma +\frac{\mathbf{x}}{2}\cdot \nabla_\Sigma -\frac{1}{2}+|A_\Sigma|^2.
    \end{equation}
 Note that a stable expander $\Sigma$ becomes strictly stable when restricting to any compact subset $K \subset \Sigma$. Denoting  $\Sigma_R : = \Sigma \cap B_R(\bOh)$ for $R \in (0,\infty)$, this implies that there exists a  positive first eigenfunction $\phi_R \in C^\infty(\Sigma_R)$ (unique up to scaling) solving 
 \begin{equation}\label{eq:first-eignfct-L-fR}
\begin{cases}
    L_\Sigma \phi_R+ \mu_R\phi_R = 0 \ &\mathrm{in} \ \Sigma_R \\
    \phi_R=0 \ &\mathrm{on} \ \partial\Sigma_R\, ,
\end{cases}
\end{equation}
where $R\mapsto \mu_R >0$ is monotonically decreasing in $R$. We will scale such that  $ \int_{\Sigma_R} \phi_R^2 e^{|\bx|^2/4} =1$, ensuring that $\phi_R$ is unique. 
 
Linearising the expander equation \eqref{eq:expander} yields solutions to the linearized equation, i.e.~functions $u\in C^\infty(\Sigma)$ such that $L_\Sigma u= 0$. We call such a function a \emph{Jacobi field}.

We further recall the following decay estimate.
\begin{prop}[{\cite[Lemma 5.3]{DingExp}}]\label{prop:DS-decay-expanders}
Let $\Sigma$ denote an expander asymptotic to a smooth cone. Then there is $R>0$ sufficiently large so that $\Sigma\setminus B_R(\bOh)$ can be written as a normal graph over $\cC$ with the graphical height function $\sigma=o(|\bx|^{-1})$ 
as $\bx\to\infty$. 
\end{prop}
This improves the trivial $\sigma=o(|\bx|)$ estimate via comparison with large spheres. 

\subsection{The level set flow of a cone and the outermost expanders} \label{subsec:level-set-flow-cone}

For a smooth cone $\cC = \cC(S)$ with $\cC=\partial W$ for $W$ a closed set, we define $\Gap(\cC)$ to be the level set flow of the cone $\cC$ at time $t=1$. Since the level set flow is unique, and $\cC$ is invariant under scaling, it follows that the level set flow of $\cC$ is given by $t \mapsto\sqrt{t}\cdot \Gap(\cC)$ for $t\in (0,\infty)$.  

The analogous statement to Proposition \ref{prop:inner-outermost-flow} holds also for the level-set flow of smooth hypercones, see \cite[Theorem E.2]{CCMS:generic1}.  Furthermore, in \cite[Theorem 8.21]{CCMS:generic1} it was shown that the outermost/innermost flows from a cone (in low dimensions) are modelled on smooth expanders, minizing the expander functional $\cE$ from the outside. (For $n=2$ smoothness had been shown by Ilmanen \cite{Trieste}.) We will refer to these as the \textit{outermost expanders}. We summarize these facts as follows:

\begin{theorem}[\cite{Trieste,CCMS:generic1}]\label{ccmsout}
 For $2\leq n\leq 6$, let $\mathcal{C}^n\subset \Rb{n+1}$ be a smooth cone. Then, there are smooth expanders $\Sigma,\Sigma'$, smoothly asymptotic to $\mathcal{C}$. The expanders $\Sigma,\Sigma'$ describe the level set flow of $\cC$ in the following sense:
 \begin{itemize}
 \item If the level set flow of $\cC$ does not fatten, then $\Gap(\cC)= \Sigma=\Sigma':=\Sigma$ . 
 \item If the level set flow of $\cC$ does fatten, then $\Sigma \cap \Sigma' = \emptyset$ and $\Gap(\cC)$ is the region between $\Sigma$ and $\Sigma'$, i.e.~$\partial\!\Gap(\cC) = \Sigma\cup\Sigma'$.
 \end{itemize}
 Finally, $\Sigma$ minimizes the expander functional $\cE$ to the outside (relative to $W$) on compact sets. Similarly, $\Sigma'$ minimizes $\cE$ to the inside. 
\end{theorem}
Note that the property that $\Sigma,\Sigma'$ minimize $\cE$ from one side implies that both $\Sigma,\Sigma'$ are stable expanders. Furthermore, for $n\geq 7$, $\Sigma,\Sigma'$ could \emph{a priori} have singular set of dimension $n-7$. We say that the \emph{outermost flows of $\cC$ are smooth} if the singular set is empty (so this always holds for $2\leq n\leq 6$). When the outermost flows of $\cC$ are \emph{a priori} known to be smooth the proof of Theorem \ref{ccmsout} carries over to prove the remaining assertions in the theorem.

Let $\Omega = W \cap \mathbb{S}^n$. Recall that $\cC = \cC(S)$ where $S= \partial \Omega \subset \mathbb{S}^n$ is a smooth embedded hypersurface. Let $\nu_S$ be the smooth unit normal vectorfield to $S$ in $\mathbb{S}^n$ pointing to the outside of $\Omega$. Given $\psi:S\to \RR$ a positive, smooth function there exist $\varepsilon>0$ and a smooth local foliation of hypersurfaces $(S_s)_{-\varepsilon<s<\varepsilon}$ in $\mathbb{S}^n$ such that $S_0 = S$ and 
$$\tfrac{\partial}{\partial s} S_s\Big|_{s=0} = \psi \cdot \nu_S \, .$$
We consider the cones $\cC_s:=\cC(S_s)$ and the corresponding outermost expanders $\Sigma_s,\Sigma_s'$. Note that by construction of the outermost flows of $\cC_s$ it follows that for $s>t$ the outermost expander $\Sigma_{s}$ lies strictly to the outside (with respect to $W$) of $\Sigma_{t}$  and $\Sigma_s \to \Sigma_t$ smoothly for $s\searrow t$. Similarly for $s<t$ the innermost expander $\Sigma'_{t}$ lies strictly to the inside of $\Sigma'_{s}$  and $\Sigma'_t \to \Sigma'_s$ smoothly for $s \nearrow t$. 

We denote with $\pi_\Gamma$ the composition of the closest point projection onto $\cC(S)$ composed with the radial projection $\cC(S) \to S$ of the cone onto its link. This is well defined on the cone over a neighborhood of $S$ in $\mathbb{S}^n$. The next lemma then follows from the above discussion together with  \cite[Lemma 2.2]{DS20} and the strong maximum principle.

\begin{lemma}\label{dsvars}
Let $\psi:S \to \RR$ be a positive, smooth function. Then there is a positive Jacobi field $v$ on $\Sigma$ 
that satisfies 
\[ |\nabla_{\Sigma}^\ell v| = O(r^{1-\ell})\]
for $\ell=0,1,2, \dots$, where $r=|x|$. Furthermore, the refined estimate
\[
v = r\cdot \psi\circ \pi_\Gamma + w
\]
with
\[
|\nabla_{\Sigma}^\ell w| = O(r^{-1-\ell})
\]
for $\ell=0,1$ holds. An analogous Jacobi field $v'$ exists on $\Sigma'$ with the same asymptotic expansions. 
\end{lemma}

\subsection{Forward rescaled flow}
Given a (smooth) mean curvature flow $(0,T) \ni t \mapsto \cM(t)$ in $\RR^{n+1}$ one obtains a solution to \emph{forward rescaled mean curvature flow} by considering the rescaling
$$ (-\infty, \log(T)) \ni \tau \mapsto \tilde M (\tau) := e^{- \tau/2} \cM(e^\tau)\, ,$$
which satisfies the evolution equation 
    \begin{align*}
    \left(\pd{\mathbf{x}}{\tau}\right)^\perp = \mathbf{H}_{\tilde M(\tau)}(\mathbf{x})-\frac{\mathbf{x}^\perp}{2}.
    \end{align*}
Note that expanders are the stationary points of this evolution.

\section{Expander barriers} \label{sec:expander-barriers}

Let $\cC = \partial W$ denote a smooth cone so that that the outermost flows of $\cC$ are modeled on smooth expanders $\Sigma,\Sigma'$. Assume that the level set flow of $\cC$ fattens (so $\Sigma$ and $\Sigma'$ are distinct). Recall that the level set flow of $\cC$ is given by $\{\sqrt{t}\Gap(\Sigma)\}_{t\in(0,\infty)}$ and $\partial\Gap(\Sigma) = \Sigma\cup\Sigma'$. Below, we work with $\Sigma$ but identical analysis for $\Sigma'$ follows by replacing $W$ with $W'=\overline{W^c}$. 

We choose the unit normal $\nu$ pointing \emph{into} $\Gap(\Sigma)$. 

By Lemma \ref{dsvars}, $\Sigma$ admits a positive Jacobi field of the form $v = r + w$ with $|\nabla^\ell w| = O(r^{-1-\ell})$. We also recall the definition of the first eigenfunction $\phi_R$ in \eqref{eq:first-eignfct-L-fR}.  For $R>0$ large and for $\alpha>0$ to be fixed sufficiently small, we define
\begin{equation}\label{eq:definition-f-alpha}
f_{3R,\alpha} = v + \alpha\phi_{3R}
\end{equation}
on $\Sigma_{3R} := \Sigma\cap B_{3R}(0)$. Then, define 
\begin{equation}\label{eq:definition-h}
h = \max_{\partial \Sigma_R} f_{3R,\alpha} > 0. 
\end{equation}
Then we define a function on all of $\Sigma$ by
\begin{equation}\label{eq:u-barrier}
u(\bx) = \begin{cases}
f_{3R,\alpha} & \bx \in \overline{\Sigma_R}\\
\min\{f_{3R,\alpha},h\} & \bx \in E_{R}\setminus \Sigma_{2R}\\
h & \bx \in E_{2R}\, ,
\end{cases}
\end{equation}
where $E_R:= \Sigma \setminus \overline{B}_{R}(0)$. We want to check that for $s>0$ sufficiently small and $\alpha,R$ chosen appropriately, the (time-independent) family of hypersurfaces $t\mapsto \Gamma_s : = \Graph_\Sigma s u$ is a  supersolution to rescaled mean curvature flow (in the sense that a rescaled mean curvature flow cannot touch $\Gamma_s$ \emph{from below} relative to its unit normal as fixed above). We start by checking that the graphs of $h$ and $f_{3R,\alpha}$ have good intersection. 

\begin{lemma}\label{lemm:subsolns-cross}
There is $R_0=R_0(\Sigma)$ so that for $R\geq R_0$ there is $\alpha_0 = \alpha_0(R,\Sigma)>0$ small so that if $\alpha \in (0,\alpha_0)$ then $h \geq f_{3R,\alpha}$ on $\partial \Sigma_R$ and $h \leq f_{3R,\alpha}$ on $\partial\Sigma_{2R}$. 
\end{lemma}
\begin{proof}
The first inequality follows from \eqref{eq:definition-h}. We now observe that (using the decay of $w$ obtained in Lemma \ref{dsvars}) 
\begin{align*}
h & = \max_{\partial\Sigma_R} f_{3R,\alpha}\\
&  \leq R + O(R^{-1}) +\alpha \max_{\partial \Sigma_R} \phi_{3R}\\
& \leq \min_{\partial\Sigma_{2R}} f_{3R,\alpha} - R + O(R^{-1}) + \alpha \left(\max_{\partial \Sigma_R} \phi_{3R} - \min_{\partial \Sigma_{2R}} \phi_{3R}\right). 
\end{align*}
Taking $R$ sufficiently large so that the second and third terms satisfy $- R + O(R^{-1}) \leq -1$ we can then take $\alpha$ sufficiently small (depending on $R$ through the dependence of $\phi_{3R}$) so the fourth term is $\leq 1$. This completes the proof. 
\end{proof}
Thus it suffices to check that $sh$ and $sf_{3R,\alpha}$ define supersolutions on the appropriate overlapping regions (for $s>0$ small). 
\begin{lemma}\label{lemm:const-h}
We can take $R=R_1(\Sigma)$ sufficiently large $h_0 =h_0(\Sigma)>0$ sufficiently small so that for $h \in (0,h_0)$, $\tau\mapsto \Gamma_{R,h} : = \Graph_{E_R} h$ defines a supersolution to rescaled mean curvature flow. 
\end{lemma}
\begin{proof}
Since $\Sigma$ is asymptotically conical, $|A| < 1$ on $E_R$ for $R$ sufficiently large. Thus, $\Gamma_{R,h}$ is smooth for $R>0$ sufficiently large. We compute
\begin{align*}
v_{\Gamma_{R,h}}(\bx,\tau) \left( \partial_\tau \bx_{\Gamma_{R,h}} - \bH + \frac 12 \bx_{\Gamma_{R,h}}\right)\cdot \nu_{\Gamma_{R,h}} &  = - L_\Sigma h + E(h)\\
& = \left( \frac 12 - |A_\Sigma|^2 \right)h + E(h)\\
& \geq \left( \frac 12 - |A_\Sigma|^2 \right)h - C_1 h^2
\end{align*}
where we used Proposition \ref{prop:expand-rescaled-mcf-app}. Taking $R$ sufficiently large so that $|A_\Sigma|^2 \leq \frac 14$. Then, taking $h_0 =\frac{1}{4C_1}$ completes the proof.
\end{proof}
We now fix $R=R_0>0$ sufficiently large so that both Lemmas \ref{lemm:subsolns-cross} and \ref{lemm:const-h} hold. Then take $\alpha_0$ as in Lemma \ref{lemm:subsolns-cross}. 
\begin{lemma}\label{lemm:subsoln-falphafR}
For any $\alpha \in (0,\alpha_0)$ there is $s_0=s_0(\Sigma,\alpha)>0$ sufficiently small, 
\[
\tau\mapsto \Gamma_{R_0,\alpha,s} : = \Graph_{\Sigma_{2R}} sf_{3R,\alpha}
\]
is a supersolution to rescaled mean curvature flow for any $s\in (0,s_0)$. 
\end{lemma}
\begin{proof}
Since $\Sigma_{2R}$ is compact, $\Gamma_{R,\alpha,s}$ will be smooth as long as $s$ is sufficiently small. Moreover, we have
\begin{align*}
v_{\Gamma_{R,\alpha,s}}(\bx,\tau) \left( \partial_\tau \bx_{\Gamma_{R,\alpha,s}} - \bH + \frac 12 \bx_{\Gamma_{R,s}}\right)\cdot \nu_{\Gamma_{R,s}} &  = - s  L_\Sigma f_{3R,\alpha} + E(s f_{3R,\alpha})\\
& = s \alpha \mu_{3R} \phi_{3R} + E(s f_{3R,\alpha}). 
\end{align*}
Now we observe that since $\phi_{3R} >0$ on $\Sigma_{3R}$, it holds that $\inf_{\Sigma_{2R}}\phi_{3R}>0$. Combined with $\mu_{3R}>0$ and the simple error estimate $ E(s f_{3R,\alpha}) = O(s^2)$ (cf.\ Proposition \ref{prop:expand-rescaled-mcf-app}), the assertion follows. 
\end{proof}

We now fix $\alpha\in(0,\alpha_0)$ and $s_0=s_0(\Sigma,\alpha)>0$ as in Lemma \ref{lemm:subsoln-falphafR}. Combining Lemmas \ref{lemm:subsolns-cross}, \ref{lemm:const-h}, and \ref{lemm:subsoln-falphafR} we obtain (recalling the definition of $u$ in \eqref{eq:u-barrier})
\begin{prop}\label{prop:fattening-barrier-final}
There is $s_0>0$ sufficiently small so that  $\tau\mapsto \Gamma_s : = \Graph_\Sigma su$ is a  supersolution to rescaled mean curvature flow for any $s\in(0,s_0)$. 
\end{prop}

\section{Fattening} \label{fattening}

We consider $M\subset \Rb{n+1}$ with an isolated singularity at $\bOh$ modeled on a smooth cone $\cC$. Let $U$ be the compact region bounded by $M$ and write $U' = \overline{U^c}$. Define $W,W' \subset \Rb{n+1}$ closed so that $\partial W =\partial W'= \cC$ and $\lim_{\rho\to\infty} \rho U = W$, $\lim_{\rho\to\infty} \rho U' = W'$ in the sense of local Hausdorff convergence. We let $\Sigma=\partial F_1(W)$ denote the outer expander and $\Sigma'=\partial F_1(W')$ the inner expander. 


Approximate $M$ to the inside and outside by smooth hypersurfaces $M_i' \subset U,M_i\subset U'$ satisfying conditions (1)-(5) described in Section \ref{sec:outermost}. Let $\cM_i,\cM'_i$ denote unit-regular cyclic Brakke flows with $\cM_i(0) = \cH^n\lfloor M_i,\cM_i'(0) = \cH^n\lfloor M_i'$ obtained via elliptic regularization. Passing to a subsequence we can assume that $\cM_i \rightharpoonup \cM,\cM_i'\rightharpoonup \cM'$ unit regular cyclic Brakke flows with $\cM(0) = \cM'(0)= \cH^n\lfloor M$.

\begin{theorem}\label{theo:fattening}
The flow $\cM$ is modeled on $\Sigma$ near $(\bOh,0)$ in the sense that $\lim_{\lambda \to \infty} \cD_\lambda(\cM) = \cM_\Sigma$. Similarly $\cM'$ is modeled on $\Sigma'$ in the sense that $\lim_{\lambda \to \infty} \cD_\lambda(\cM') = \cM_{\Sigma'}$.
\end{theorem}
In particular, this shows that if $\cC$ fattens under the level-set flow then so does $M$.

Before proving Theorem \ref{theo:fattening} we observe the following consequence. Recall that $M(t)$ (resp.\ $M'(t)$) is the outer (resp.\ inner) flow of $M$ as defined in Proposition \ref{prop:inner-outermost-flow}.  
\begin{corollary}\label{coro:smoothness}
There is $T>0$ so that $\cM\lfloor\{t<T\}$ and $\cM'\lfloor\{t<T\}$ are smooth and for $0\leq t <T$, we have $\supp \cM(t) = M(t)$ (resp.~$\supp \cM'(t) = M'(t)$). Furthermore, any unit regular integral Brakke flow $\check\cM$ with $\check\cM(0) = \cH^n\lfloor M$ and $\supp\check\cM(t) \subset M(t)$ satisfies $\check\cM\lfloor\{t<T\} = \cM\lfloor \{t<T\}$.\end{corollary}
\begin{proof}
It suffices to consider $\cM$ and outer flows. Suppose there are points $X_i = (\bx_i,t_i) \in \sing \cM$ with $0<t_i\to0$. Since $M$ is smooth away from $\bOh$, it must hold that $\bx_i\to\bOh$. Suppose that up to a subsequence $\supp_i |\bx_i|^2 t_i^{-1} <\infty$. Then by Theorem \ref{theo:fattening}, $\cD_{t_i^{-1/2}}(\cM) \rightharpoonup \cM_{\Sigma}$ as $i\to\infty$. Since $\cD_{t_i^{-1/2}}(X_i) = (\bx_i t_i^{-1/2},1)$ is bounded (by our assumption), this contradicts the fact that $\Sigma$ is smooth. Thus, it remains to consider the case that $|\bx_i|^2 t_i^{-1} \to\infty$. By Theorem \ref{theo:fattening} again, $\cD_{|\bx_i|^{-1}}(\cM)\rightharpoonup \cM_{\Sigma}$. Up to a subsequence, $\cD_{|\bx_i|}(X_i) = (|\bx_i|^{-1} \bx_i ,|\bx_i|^{-2} t_i)$ converges to $(\tilde \bx,0)$ with $|\tilde \bx| = 1$. Since $\sing\cM_\Sigma = (\bOh,0)$, this is a contradiction. This proves the smoothness part of the assertion. 

By the work of Hershkovits--White as recalled in Proposition \ref{prop:inner-outermost-flow}, the support of $\cM$ agrees with the outer flow of $M$. The final statement follows since $t\mapsto M(t)$ is a smooth mean curvature flow for $0<t<T$, so the constancy theorem implies that the multiplicity of $\check \cM$ is a non-negative constant for a.e.~$0<t<T$, which additionally is monotone in time. But the monotonicity formula together with unit regularity implies that the multiplicity is one away from $(\bOh,0)$, so $\check \cM$ agrees with $\cM$.
\end{proof}

\begin{proof}[Proof of Theorem \ref{theo:fattening}]
It suffices to consider the outer flow $\cM$. Fix $s\in (0,s_0)$ and let $\Gamma_s$ be defined as in Proposition \ref{prop:fattening-barrier-final} with respect to $\Sigma$. Recall that $t\mapsto \sqrt{t} \Gamma_s$ is a supersolution to mean curvature flow and that the unit normal to $\Sigma$ points into $\Gap(\cC)$.  

Using the notation established in the beginning of this section, we have:

\begin{claim}\label{claim:blowup-outer-aprooxrs}
For $i(j)\to\infty$ and $\lambda_j\to\infty$ assume that $\tilde\cM_j:=\cD_{\lambda_j}(\cM_{i(j)}) \rightharpoonup \tilde\cM$. Then $\supp \tilde\cM(t) \subset W'(t)$ (the level set flow of $W'$). 
\end{claim}
\begin{proof}
By construction, $\supp\tilde\cM(0) \subset W'$. Thus, the assertion follows from the avoidance principle for Brakke flows (cf.\ \cite[Theorem 10.7]{ilmanen}). \end{proof}

\begin{claim}\label{claim:boundary-true-1}
There is $\rho_0=\rho_0(s) <\infty$ so that $W'(1) \setminus B_{\rho_0}(\bOh)$ lies strictly below $\Gamma_s$. 
\end{claim}

\begin{proof}
Recall that $W'(1)$ is the region below $\Sigma'$. Thus, the assertion follows from the fact that $\Sigma$ decays towards $\Sigma'$ (cf.\ Proposition \ref{prop:DS-decay-expanders}) whereas the function  $u$ in \eqref{eq:u-barrier} is constant near infinity.
\end{proof}

\begin{claim} \label{claim:boundary-true-2}
Fix $\rho\geq \rho_0$. There's $T=T(s,\rho)$ so that $\supp\cM_i(t) \cap B_{\sqrt{t}\rho}(\bOh)$ lies below $\sqrt{t}\Gamma_s$ for all $t \in (0, T]$ and $i \in \NN$.
\end{claim}

\begin{proof}[Proof of Claim \ref{claim:boundary-true-2}]
We first show that there $T=T(s,\rho)>0$ such that the claim holds in a neighbourhood of $\partial B_{\sqrt{t}\rho}(\bOh)$. Assume this fails at times $0<t_j\to0$ for the flows $\cM_{i(j)}$. Note that $i(j) \to \infty$ as $j\to \infty$ since $M_i$ is disjoint from $M$. Let $\tilde\cM_j : = \cD_{t_j^{-1/2}}(\cM_{i(j)})$. Pass to a subsequence so that $\tilde\cM_j\rightharpoonup \tilde\cM$. By Claim \ref{claim:blowup-outer-aprooxrs}, $\supp\tilde \cM(t) \subset W'(t)$. However, by Claim \ref{claim:boundary-true-1}, $W'(t)\setminus B_{\sqrt{t}\rho_0}(\bOh)$ lies strictly below $\sqrt{t}\Gamma_s$ for $t >0$. Combined with upper semi-continuity of Gaussian density along Brakke flow convergence (implying that points in the support of a converging sequence of flows converge to points in support) we obtain a contradiction. 

Now, we note that $\supp \cM_i(t)$ is disjoint from $B_{\sqrt{t}\rho}$ for all $0 \leq t \leq 2t(i)$ (since $M_i$ is disjoint from $\bOh$). The claim then follows by applying the maximum principle in $B_{\sqrt{t}\rho}(\bOh)$ from $t(i)$ to $T(s,\rho)$.
\end{proof}

Letting $i \to \infty$ establishes the same statement as in Claim \ref{claim:boundary-true-2} for with $\cM_i$ replaced by $\cM$. Since the flow $t \mapsto\sqrt{t}\Gamma_s$ is scaling invariant this implies that any forwards blow-up of $\cM$ at $(\bOh,0)$ has to lie (weakly) below $\sqrt{t}\Gamma_s\cap B_{\sqrt{t}\rho}(\bOh)$ for all $t >0$. Letting $\rho \to \infty$ and $s\to 0$ establishes that any forward blow-up\footnote{i.e.\ any subsequential limit of $\cD_\lambda(\tilde\cM)$ as $\lambda\to\infty$} of $\tilde \cM$ is supported (weakly) below $t\mapsto \sqrt{t} \Sigma$, starting at $\cC$. Since $\Sigma$ is the outer expander, the support of $\tilde\cM$ thus must be a subset of $t\mapsto \sqrt{t} \Sigma$. Again, the constancy theorem implies that the multiplicity is a non-negative constant for a.e.~$t>0$, which additionally is monotone in time. But the monotonicity formula together with unit regularity implies that the multiplicity is one sufficiently far out, so any blow-up of $\cM$ agrees with $t\mapsto \sqrt{t} \Sigma$. 
\end{proof}

\section{Global barriers} \label{sec:glob-barrier}

Let $\cC$ denote a smooth cone over its link $S \subset \mathbb{S}^{n}$ and  $\Sigma$ a smooth, stable expander asymptotic to $\cC$. We choose a global unit normal vector field  on $\nu_\Sigma$.  As in Lemma \ref{dsvars} for the outermost expanders, we assume that $\Sigma$ admits a positive Jacobi field of the form $v = r \cdot \psi \circ \pi_\Gamma + w$, where $\psi$ is a positive function on $\Gamma$, together with the decay $|\nabla^\ell w| = O(r^{-1-\ell})$ for $l = 0,1$.

We mimic part of the construction in Section \ref{sec:expander-barriers}. Since $\Sigma$ is stable, for every $R>0$ large we have a positive first eigenfunction $\phi_{3R}$ in \eqref{eq:first-eignfct-L-fR} with eigenvalue $\mu_{3R}>0$.  For $\alpha>0$ to be fixed sufficiently small, we take
\begin{equation}
f_{3R,\alpha} = v + \alpha\phi_{3R}
\end{equation}
on $\Sigma_{3R} := \Sigma\cap B_{3R}(\bOh)$. Then, define 
\begin{equation}
h = \max_{\partial \Sigma_R} f_{3R,\alpha} > 0. 
\end{equation}
We now consider $M\subset \Rb{n+1}$ with an isolated singularity at $\bOh$ modelled on $\cC$. Assume that $\cM$ is a unit-regular cyclic Brakke flow, smooth on $(0,T)$ for some $T>0$ with $\cM(0) = \cH^n\lfloor M$ and such that $t^{-1/2} \cM(t)$ converges smoothly to $\Sigma$ (on compact subsets of $\Rb{n+1}$) as $t\searrow 0$. Thus for every $\delta>0$ and every $R<\infty$ there is $T_{\delta,R}$ and a smooth family of functions $u:\Sigma_{3R} \times (0, T_{\delta,R})$ such that for all $t\in (0, T_{\delta,R})$
\begin{equation} \label{eq:graphical}
  \{ \bx + u(\bx, t) \nu_\Sigma(\bx)\, | \, \bx \in \Sigma_{3R}\} \subset t^{-1/2} \cM(t)
\end{equation}
with $\|u(\cdot, t)\|_{C^3(\Sigma_{3R})} \leq \frac \delta 2$. For $s\in (0,1)$ we define the \emph{close barriers}
$$ \Gamma^\pm_{\textnormal{close},s}(t) = \sqrt{t}\cdot \{  (\bx + (u(\bx, t)\pm s f_{3R,\alpha})\cdot  \nu_\Sigma(\bx))\, | \, \bx \in \Sigma_{2R}\}$$
and the \emph{far barriers}
$$ \Gamma^\pm_{\textnormal{far},s}(t) = \{  (\bx \pm s h \sqrt{t} \cdot \nu_{\cM(t)}(\bx))\, | \, \bx \in \cM(t)\setminus B_{\sqrt{t} R}(\bOh)\}\, ,$$
where we choose the unit normal vectorfield $\nu_{\cM(t)}$ such that it induces the same orientations as $\nu_\Sigma$ in the convergence $t^{-1/2} \cM(t) \to \Sigma$ as $t\searrow 0$. 

We aim to check that for $s>0$ sufficiently small and $\alpha,R$ chosen appropriately 
$t\mapsto \Gamma^+_{\textnormal{close/far}, s}(t)$  constitute supersolutions to mean curvature flow (in the sense that a mean curvature flow cannot touch $t\mapsto \Gamma^+_{\textnormal{close/far}, s}(t)$ \emph{from below} relative to its unit normal as fixed above) away from their respective boundaries. Similarly, $t\mapsto \Gamma^-_{\textnormal{close/far}, s}(t)$  constitute  subsolutions to mean curvature flow (in the sense that a mean curvature flow cannot touch $t\mapsto \Gamma^-_{\textnormal{close/far}, s}(t)$ \emph{from above} relative to its unit normal as fixed above) away from their respective boundaries.

To construct global barriers, we start by checking that $\Gamma^\pm_{\textnormal{close},s}$ and $\Gamma^\pm_{\textnormal{far},s}$ have good intersection. This will be used to ``weld'' them together in the sense of Meeks--Yau \cite{MeeksYau} to form a global barrier. 

\begin{lemma}\label{lemm:subsolns-cross-v2}
There is $R_0=R_0(\Sigma)$  so that for $R\geq R_0$ there is $\alpha_0 = \alpha_0(R,\Sigma)>0$ small and  $\delta_0 =\delta_0(R,\Sigma, \alpha_0)>0$ and so that if $\alpha \in (0,\alpha_0)$ and $\delta \in (0,\delta_0)$ then for $t\in (0, T_{\delta, R})$ we have that
\begin{itemize}
 \item[(i)] $\Gamma^+_{\textnormal{close},s}(t)$ lies above $ \Gamma^+_{\textnormal{far},s}(t)$ in a neighborhood of $\partial\Gamma^+_{\textnormal{close},s}(t)$.
  \item[(ii)] $\Gamma^+_{\textnormal{far},s}(t)$ lies above $ \Gamma^+_{\textnormal{close},s}(t)$ in a neighborhood of $\partial\Gamma^+_{\textnormal{far},s}(t)$.
  \item[(i)] $\Gamma^-_{\textnormal{close},s}(t)$ lies below $ \Gamma^-_{\textnormal{far},s}(t)$ in a neighborhood of $\partial\Gamma^-_{\textnormal{close},s}(t)$.
  \item[(ii)] $\Gamma^-_{\textnormal{far},s}(t)$ lies above $ \Gamma^-_{\textnormal{close},s}(t)$ in a neighborhood of $\partial\Gamma^-_{\textnormal{far},s}(t)$.
\end{itemize}
for all $s\in (0,1)$.
\end{lemma}
\begin{proof}
This follows from Lemma \ref{lemm:subsolns-cross} by choosing $\delta_0>0$ sufficiently small.
\end{proof}

\begin{lemma}\label{lemm:inner-barrier}
For $R\geq R_0(\Sigma)$, $\alpha \in (0,\alpha_0)$, there is $\delta_1=\delta_1(R,\alpha,\Sigma)>0$ and $s_0=s_0(R,\Sigma)>0$, such that for all $s\in (0,s_0)$ and $\delta \in (0,\delta_1)$, for $t \in (0,T_{\delta,R})$,
\[
t\mapsto \Gamma_{\textnormal{close},s}^+(t)
\]
is a supersolution to mean curvature flow (and similarly  $t\mapsto \Gamma_{\textnormal{close},s}^-(t)$ is a subsolution to mean curvature flow).
\end{lemma}
\begin{proof}
By rescaling it is equivalent to show that 
$$\tau \mapsto \tilde \Gamma(\tau):= e^{-\tau/2} \Gamma_{\textnormal{close},s}^+(\ln(\tau))$$ 
for $\tau \in (-\infty, \exp(T_{\delta,r}))$ is a supersolution to rescaled mean curvature flow. Similarly we denote the corresponding  solution of $\cM$ to rescaled mean curavture flow by
$$\tau \mapsto \tilde\cM(\tau):= e^{-\tau/2}\cM (\ln(\tau))\, ,$$
which by \eqref{eq:graphical} (writing $\tilde u(\bx, \tau) := u(\bx, \ln{\tau})$) satisfies
 $$\{ \bx + \tilde u(\bx, \tau) \nu_\Sigma(\bx)\, | \, \bx \in \Sigma_{3R}\} \subset  \tilde \cM(\tau)\, .$$
Applying Lemma \ref{lemm:relative-expander-mean-curvature} we compute
\begin{equation}\label{eq:linearisation_1}
 v_{\tilde \Gamma} \left( \partial_\tau \bx_{\tilde \Gamma} - \bH_{\tilde \Gamma} + \frac 12 \bx_{\tilde \Gamma}\right)\cdot \nu_{\tilde \Gamma}   = - sL_\Sigma f_{3R,\alpha}  + E^{sf_{3R,\alpha}} = \alpha s \mu_{3R} \phi_{3R}+E^{sf_{3R,\alpha}}
\end{equation}
and provided $\|\tilde u\|_{C^3(\Sigma_{2R})} + s \|f_{3R,\alpha}\|_{C^3(\Sigma_{2R})} \leq \delta$ we can estimate pointwise
$$|E^{sf_{3R,\alpha}}(\bx,\tau)| \leq Cs \delta (|f_{3R,\alpha}(\bx)| + |\nabla_\Sigma f_{3R,\alpha}(\bx)|+ |\nabla^2_\Sigma f_{3R,\alpha}(\bx)|)\, .$$
Note that $\phi_{3R} >0$ on $\Sigma_{3R}$, so there exists $C>0$ such that for $\bx \in \Sigma_{2R}$
$$ |f_{3R,\alpha}(\bx)| + |\nabla_\Sigma f_{3R,\alpha}(\bx)|+ |\nabla^2_\Sigma f_{3R,\alpha}(\bx)| \leq C \phi_{3R}(\bx)\, .$$
Thus \eqref{eq:linearisation_1} yields (for all $s\in (0,s_0)$)
$$ v_{\tilde \Gamma} \left( \partial_\tau \bx_{\tilde \Gamma} - \bH_{\tilde \Gamma} + \frac 12 \bx_{\tilde \Gamma}\right)\cdot \nu_{\tilde \Gamma}   \geq s(\alpha \mu_{3R} - C\delta) \phi_{3R} > 0\, ,$$
on $\Sigma_{2R}$, as long as $\delta$ is sufficiently small. This yields the statement for $t\mapsto  \Gamma_{\textnormal{close},s}^+(t)$. The statement for $t\mapsto  \Gamma_{\textnormal{close},s}^-(t)$ follows analogously.
\end{proof}

Take $R_0$ as in Lemma \ref{lemm:subsolns-cross-v2}. 
\begin{lemma}\label{lemm:outer-barrier}
There is $R=R_1(M,\Sigma)\geq R_0$ sufficiently large so that for $R\geq R_1$, taking $\alpha_0$ as in Lemma \ref{lemm:subsolns-cross-v2}, $s_0$ as in Lemma \ref{lemm:inner-barrier} and $T_\textnormal{far}>0$ sufficiently small depending on $M,\Sigma,R$ the following holds: for $t\in (0, T_\textnormal{far})$ the flow $t\mapsto \Gamma_{\textnormal{far},s}^+(t)$ is a supersolution to mean curvature flow, away from its boundary (and similarly  $t\mapsto \Gamma_{\textnormal{far},s}^-(t)$ is a subsolution to mean curvature flow, away from its boundary).
\end{lemma}
\begin{proof} 
We first establish the following claim. Consider the $n$-dimensional ball in the $\{x_{n+1} =0\}$ hyperplane, given by $B^n_r(0):= B_r(0) \cap \{x_{n+1} =0\}$.
\begin{claim}\label{claim:pseudo-barrier} There exists $\eta=\eta(n)>0$ with the following property.
 Let $u: B^n_1(0) \times [0,1]$ with $\|u(\cdot, t)\|_{C^3(B^n_1(\bOh))} \leq \eta$ such that for $t\in [0,1]$
 $$ \cM(t) := \{(\hat \bx, u(\hat \bx,t))\, |\, \hat \bx \in B^n_1(0)\}$$
 constitutes a smooth mean curvature flow. Let $\nu_\cM(\cdot, t)$ be the upwards unit normal to $\cM(t)$. Consider
 $$\Gamma^\pm(t):= \{\bx \pm s \sqrt{t} \nu_\cM(\bx, t)\, |\, \bx \in \cM(t)\}\, .$$
 Then for $s\in (0,1)$, the flow $(0,1] \ni t\mapsto \Gamma^+(t)$ is a supersolution to mean curvature flow. (Similarly, the flow $(0,1] \ni  t\mapsto \Gamma^-(t)$ is a subsolution to mean curvature flow.)
\end{claim}
\begin{proof}[Proof of Claim]
 We consider $\Gamma^+(t)$ (the computation for $\Gamma^-(t)$ is analogous). Let $t\mapsto \bx_\cM(t)$ be a point evolving normally, i.e.\ $\partial_t\bx_\cM = \bH_\cM(\bx_\cM)$. Recall that 
 \[
\left|\partial_t \nu_\cM(\bx_\cM(t),t)\right| = |\nabla_\cM H_\cM| \leq C \eta. 
 \]
 Thus if we set $\bx_{\Gamma^+} = \bx_\cM + s\sqrt{t}\nu_\cM$ then
 \[
 \partial_t\bx_{\Gamma^+} = \bH_{\cM} + (\partial_t(s\sqrt{t}))\nu_\cM + s\sqrt{t} (\partial_t \nu_\cM)
 \]
 so this point is evolving with normal speed
 \[
 v_{\Gamma^+} (\partial_t\bx_{\Gamma^+}) \cdot \nu_{\Gamma^+} = H_\cM + \partial_t(s\sqrt{t}) + s\sqrt{t} (\partial_t \nu_\cM) \cdot \nu_{\Gamma^+} 
 \]
 where $v_{\Gamma^+} = (\nu_{\Gamma^+}(\bx_{\Gamma^+}) \cdot \nu_\cM(\bx_\cM))^{-1}$ (cf. Appendix \ref{app:graphical-eqns-defn-v}). Thus
 \begin{align*}
 v_{\Gamma^+} (\partial_t \bx_{\Gamma^+} - \bH_{\Gamma^+} )  \cdot \nu_{\Gamma^+} & = H_\cM - v_{\Gamma^+} H_{\Gamma^+} + \partial_t(s\sqrt{t}) + s\sqrt{t} (\partial_t \nu_\cM) \cdot \nu_{\Gamma^+} \\
 & = \partial_t(s\sqrt{t}) - |A_\cM|^2 s\sqrt{t} + E\, ,
 \end{align*}
where the error satisfies $ |E| \leq  C s\eta \sqrt{t}$, where $C$ depends only on $n$ (using Lemma \ref{lemm:expand-H-app} to estimate the difference in mean curvatures and the computation above to estimate the unit normal evolution term). Thus, for $\eta> 0$ sufficiently small and $t\in (0,1]$
 \[
  v_{\Gamma^+} \left( \partial_\tau \bx_{\Gamma^+} - \bH_{\Gamma^+}\right)\cdot \nu_{\Gamma^+}  \geq (\tfrac{1}{2t} - C\eta)  s\sqrt{t} \geq (\tfrac{1}{2} - C\eta)  s\sqrt{t} > 0\, .
  \]
 This establishes the claim.
 \end{proof}
There is $C_M>0$ so that $|\by| |A_M(\by)|+ |\by|^2 |\nabla A_M(\by)| \leq C_M $ (this follows because $M$ is modeled on a smooth cone at $\bOh$). This implies that after choosing $R$ sufficiently large, we have for all $\bx \in M\setminus \{\bOh\}$ that
\[
(R|\bx|^{-1}(M - \bx))\cap B_2(\bOh)
\]
can be written as a $C^3$-graph over its tangent plane with $C^3$-norm bounded by $O(R^{-1})$. By pseudolocality (cf.\ \cite[Theorem 1.5]{INS} or \cite[Theorem 7.3]{CY:pseudo}) this implies that (taking $R$ sufficiently large)
\[
(R|\bx|^{-1}(\cM(R^{-2}|\bx|^2t) - \bx)) \cap B_1(\bOh)
\]
remains a small Lipschitz graph over its tangent plane for $t \in [0,2]$. Parabolic estimates then imply that if we take $R$ even larger, the graph will have $C^3$-norm bounded by $\eta$ (as defined in Claim \ref{claim:pseudo-barrier}). 

We can thus apply the scaled version of Claim \ref{claim:pseudo-barrier} over balls at $\bx \in M$ of radius $R^{-1}|\bx|$. The claim only applies for times $t \in [0,R^{-2}|\bx|^2]$, but for larger times, the graph in $\cM(t)$ over this ball will be contained in $B_{\sqrt{t}R}$, and thus does not contribute to $\Gamma^+_{\textrm{far},s}(t)$. This completes the proof.
\end{proof}
We now choose $R\geq R_1$ to satisfy Lemma \ref{lemm:outer-barrier}. We then choose $\alpha \in (0,\alpha_0)$ as in Lemma \ref{lemm:subsolns-cross-v2} and then $s_0$ as in Lemma \ref{lemm:inner-barrier}. 

\begin{prop}\label{prop:uniqueness-barrier}  There is $\delta>0$ sufficiently small such that for all $s\in (0,s_0)$ and $t\in (0, T_{\delta, R})$,  we can weld $t\mapsto \Gamma_{\textnormal{close},s}^+(t)$ to $t\mapsto \Gamma_{\textnormal{far},s}^+(t)$ to obtain a global supersolution $t\mapsto \Gamma_{s}^+(t)$ to mean curvature flow. Similarly, we can weld $t\mapsto \Gamma_{\textnormal{close},s}^-(t)$ to $t\mapsto \Gamma_{\textnormal{far},s}^-(t)$ to obtain a global subsolution $t\mapsto \Gamma_{s}^-(t)$ to mean curvature flow.
\end{prop}

\section{Uniqueness} \label{sec:unique}

We work with the same set-up as in the previous section.

Assume that we have smooth mean curvature flows $\cM^1$ and $\cM^2$, defined on $(0,T)$ for some $T>0$, with $\cM^1(0)=\cM^2(0) = \cH^n\lfloor M$ and such that $t^{-1/2} \cM^i(t)$ converges smoothly to $\Sigma$ (on compact subsets of $\Rb{n+1}$) as $t\searrow 0$ for $i=1,2$.

\begin{lemma}\label{lem:distance-estimate}
 It holds that $\textnormal{dist}(\cM^1(t), \cM^2(t)) = o(\sqrt{t})$.
\end{lemma}
\begin{proof}
Assume that there is $\kappa >0$ and $\bx_i \in \supp \cM^1(t_i)$ with $t_i\to 0$ so that 
\begin{equation}\label{eq:dist-est-cont}
d(\bx_i,\supp\cM^2(t_i)) \geq \kappa \sqrt{t_i}.
\end{equation}
Pass to a subsequence so that $\bx_i\to\bx \in M$. If $\bx=\bOh$ then $(t_i)^{-1/2} \cM^i(t_i)$ both converge to $\Sigma$ in $C^\infty_\textrm{loc}(\RR^{n+1})$. If $\bx\neq \bOh$ then short-time smoothness of $\cM^i(t)$ around $\bx$ gives that\footnote{Note that one can view the $\bx\neq \bOh$ case as the ``same'' as the $\bx=\bOh$ case, since $M$ is modeled on the cone $T_\bx M$ at $\bx\neq \bOh$, which evolves as a static hyperplane under level set flow.} $(t_i)^{-1/2} (\cM^i(t_i)-\bx)$ both converge to $T_\bx M$ in $C^\infty_\textrm{loc}(\RR^{n+1})$. In either case, we see that assuption \eqref{eq:dist-est-cont} cannot hold. This completes the proof.  
 \end{proof}

\begin{prop}\label{prop:uniqueness}
It holds that $\cM^1(t) = \cM^2(t)$ on $[0,T)$.
\end{prop}
\begin{proof}
By Proposition \ref{prop:uniqueness-barrier} we can construct for all $s\in (0,1)$ and $t\in (0, T_{\delta, R})$ the supersolution $t\mapsto \Gamma_{s}^+(t)$ and subsolution $t\mapsto \Gamma_{s}^-(t)$ over $\cM^1$. Note that for any $s\in (0,s_0)$, we have by Lemma \ref{lem:distance-estimate} that $\cM^2(t)$ lies between $\Gamma_s^+(t)$ and $\Gamma_s^-(t)$ for $t>0$ sufficiently small. This yields for all $s\in (0,s_0)$ and $t\in (0, T_{\delta, R})$ that $\cM^2(t)$ lies between $\Gamma_s^+(t)$ and $\Gamma_s^-(t)$. Since for all $t\in (0, T_{\delta, R})$ we have that $\Gamma_s^+(t), \Gamma_s^-(t) \to \cM^1(t)$ as $s\to 0$ this yields that $\cM^1$ coincides with $\cM^2$ on $(0, T_{\delta, R})$. Thus $\cM^1$ and $\cM^2$ have to coincide (at least as long as they both remain smooth).
\end{proof}

\begin{corollary}\label{coro:non-fat}
If $\cC$ does not fatten under the level set flow, then the inner and outer flows agree for $t \in [0,T]$.
\end{corollary}
\begin{proof}
If $\cC$ does not fatten then $\Sigma=\Sigma'$. Thus, by Theorem \ref{theo:fattening}, the outer and inner flows $\cM,\cM'$ are both modeled on $\Sigma$ near $(\bOh,0)$. The assertion then follows from Proposition \ref{prop:uniqueness}. 
\end{proof}

\begin{appendix} 
\section{Graphs over expanders}\label{sec:graphs}
We consider a smooth embedded hypersurface $M \subset \RR^{n+1}$ and assume that $\nu_M$ is a choice of smooth unit normal vector field to $M$. We define its \emph{shape operator} (or Weingarten map)
\begin{equation}\label{eq:convent.shape.oper}
S_p: T_pM \to T_pM, \qquad \xi \mapsto - D_\xi\nu_\Sigma 
\end{equation}
and its second fundamental form
\begin{equation}\label{eq:convent.sff}
 A: T_pM \times T_pM \to \RR,\qquad (\xi,\zeta) \mapsto A(\xi,\zeta) = S_p(\xi) \cdot \zeta\, .
 \end{equation}
We fix the sign of the \emph{scalar mean curvature} $H$ as follows
$$ \bH = H\, \nu_\Sigma\, ,$$
and thus $H = \tra S  = \tra A $, with the principal curvatures of $M$ being the eigenvalues of $S$. We  consider  
 $u : M\times I\to \RR$ so that
\[
|u||A_M| < \eta < 1
\]
along $M\times I$, where $A$ is the second fundamental form of $\Sigma$. This allows us to define the graph
\[
\Gamma_\tau  : =\{ \bx + u(\bx,\tau) \nu_M(\bx): \bx \in M\}. 
\]
We compute here various geometric quantities associated to $\Gamma_\tau$. The computations follow directly as in \cite[Appendix A]{ccs23}. There the focus was on the backwards rescaled flow, but the computations for the forwards rescaled flow are completely analogous and just amount to changing sign in front of the corresponding terms.

Define
\begin{equation}\label{app:graphical-eqns-defn-v}
v(\bx, \tau) = (1+|(\Id - u S_M)^{-1}(\nabla_M u)|^2)^{\frac 12}. 
\end{equation}
\begin{lemma}[{\cite[Lemma A.1]{ccs23}}]
The upwards pointing normal along $\Gamma_\tau$ is
\begin{equation}\label{eq:app-unit-normal-graph}
\nu_{\Gamma} = v^{-1}(-(\Id - u S_M)^{-1}\nabla_M u + \nu_M).
\end{equation}
In particular
\begin{equation}\label{eq:app-defi-vdotv-v}
v = (\nu_{\Gamma} \cdot \nu_M)^{-1}. 
\end{equation}
\end{lemma}

For $\ell \in \{0,1\}$ set
\begin{equation}\label{eq:defn-sigma-x-app}
\sigma_{\ell}(\bx,\tau) : = \sum_{j=0}^\ell \sum_{k=0}^{4-2j} |\partial_\tau^j\nabla^k u(\bx,\tau)|. 
\end{equation}

\begin{lemma}[{\cite[Lemma A.2]{ccs23}}]\label{lemm:expand-H-app}
The mean curvature of $\Gamma_\tau$ at $\bx + u(\bx,\tau) \nu_M(\bx)$ satisfies
\begin{equation}
v(\bx,\tau) H_{\Gamma}(\bx + u(\bx,\tau) \nu_M(\bx),\tau) = H_M(\bx) +  (\Delta_M u + |A_M |^2 u)(\bx) + E^H
\end{equation}
where the error $E^H$ can be decomposed into terms of the form 
\[
E^H = u E^H_1  + E^H_2(\nabla_M u,\nabla_M u) 
\]
where $E^H_1 \in C^\infty(\Sigma)$ and $E^H_2 \in C^\infty(\Sigma ; T^*M \otimes T^*M)$ satisfy the following estimates: 
\[
|\partial_\tau  E_1^H(\bx,\tau)| \leq C_1^H\sigma_1(\bx,\tau),\qquad  \sum_{k=0}^2 |\nabla^k_\Sigma E_1^H(\bx,\tau)|  \leq C_1^H \sigma_0(\bx,\tau)
\]
and\footnote{recall that $E_2$ is a section of $T^*\Sigma\otimes T^*\Sigma$ so e.g., $\nabla_\Sigma E_2$ is a section of $T^*\Sigma\otimes T^*\Sigma \otimes T^*\Sigma$}
\[
|\partial_\tau  E_2^H(\bx,t)| \leq C_2^H (1+\sigma_1(\bx,\tau)),\qquad \sum_{k=0}^2 |\nabla^k E_2^H(\bx,t)|  \leq C_2^H (1+ \sigma_0(\bx,\tau))
\]
where $C_1^H,C_2^H$ depend only on $\eta$ and an upper bound for $\sum_{k=0}^3 |\nabla^k A|(\bx)$. 
\end{lemma}

Observe that the expander mean curvature can be written as
\[
\bH - \frac{\bx^\perp}{2} = \left(H - \frac 12  \bx \cdot \nu_\Gamma \right) \nu_\Gamma\, 
\]
and recall the definition of the (expander) Jacobi operator $L$, see \eqref{eq:definiton-Jacobi}. For the following we assume that $M = \Sigma$, where $\Sigma$ is an (open subset of an) expander.

\begin{prop}[{\cite[Corollary A.4]{ccs23}}]\label{prop:expand-rescaled-mcf-app}
We have
\begin{align*}
 v(\bx,\tau) (\partial_{\tau}\bx_{\Gamma} - \bH +  \tfrac 12 \bx_{\Gamma}) \cdot \nu_{\Gamma} & = \partial_{\tau}u  - {\big (\underbrace{\Delta u + \tfrac 12 \bx^T \cdot \nabla u + ( |A|^2 - \tfrac 12) u}_{=L u} \big) }+ E 
\end{align*}
at $\bx_{\Gamma} = \bx + u(\bx,\tau)\nu_\Sigma(\bx)$ for  $E=uE_1 + E_2(\nabla u,\nabla u)$ for $E_1,E_2$ satisfying
\[
|\partial_\tau  E_1(\bx,\tau)| \leq C_1\sigma_1(\bx,\tau), \qquad \sum_{k=0}^2 |\nabla^k E_1(\bx,\tau)|  \leq C_1\sigma_0(\bx,\tau)
\]
and
\[
|\partial_\tau  E_2(\bx,\tau)| \leq C_2(1+\sigma_1(\bx,\tau)),\qquad \sum_{k=0}^2 |\nabla^k E_2(\bx,\tau)|  \leq C_2 (1+ \sigma_0(\bx,\tau))
\]
where $C_1,C_2$ depend only on $\eta$ and an upper bound for $\sum_{k=0}^3 |\nabla^k_\Sigma A|(\bx)$. 

In particular, if $\Gamma_\tau$ is a solution to rescaled mean curvature flow, i.e.,
\[
(\partial_\tau \bx)^\perp = \bH - \tfrac 12\bx^\perp
\]
then 
\begin{equation}\label{eq:error-term-linearize}
\partial_\tau u = Lu + E
\end{equation}
for $E$ as above. 
\end{prop}

We also need to understand the linearization of the expander mean curvature of two graphs over $\Sigma$, relative to each other and relative to the base $\Sigma$.

\begin{lemma}[{\cite[Lemma A.7]{ccs23}}] \label{lemm:relative-expander-mean-curvature} For $\delta < \tfrac 12 (\sup_\Omega |A_\Sigma|)^{-1}$ and  $u_i \in C^\infty(\Sigma), i = 0,1$ with 
\begin{equation}\label{eq:C2-bound}
 |u_i(\bx)|+ |\nabla_\Sigma u_i(\bx)| + |\nabla^2_\Sigma u_i(\bx) | + |\nabla^3_\Sigma u_i(\bx) | \leq \delta
\end{equation}
for all $\bx \in \Sigma$. Letting
 $$\Gamma_i:= \{ \bx + u_i(\bx) \nu_\Sigma(\bx): \bx \in \Sigma\}$$
 then denoting $w =u_1-u_0$ and $v_i:= (\nu_{\Gamma_i} \cdot \nu_\Sigma)^{-1}$ there exists $C=C(\sup_{\Sigma} |A_{\Sigma}| + |\nabla A_{\Sigma}|+ |\nabla^2 A_{\Sigma}|)$ such that
 \begin{equation*}
  v_1\Big(H_{\Sigma_1} - \tfrac 12 \bx_{\Sigma_1} \cdot \nu_{\Sigma_1}\Big) - v_0\Big(H_{\Sigma_0}- \tfrac 12\bx_{\Sigma_0}\cdot \nu_{\Sigma_0}\Big)  
  = L_\Sigma w + E^w
 \end{equation*}
 where $\bx_{\Sigma_i} = \bx + u_i(\bx)\nu_\Sigma(\bx)$ and the error term $E$ satisfies
 $$ E^w(\bx) = w(\bx) F(\bx) + \nabla w(\bx) \cdot \bF(\bx) + \nabla^2 w(\bx) \cdot \mathcal{F}(\bx) \, ,$$
 with the estimate
 $$|F| + |\bF| + |\cF| + |\nabla \cF| \leq C\delta$$
 for all $\bx \in \Sigma$.  \end{lemma}

\end{appendix}

\bibliographystyle{alpha}
\bibliography{UniEvo}

\begin{thebibliography}{CCMS24b}

\bibitem[AAG95]{AltAngGig}
Steven Altschuler, Sigurd~B. Angenent, and Yoshikazu Giga.
\newblock Mean curvature flow through singularities for surfaces of rotation.
\newblock {\em J. Geom. Anal.}, 5(3):293--358, 1995.

\bibitem[AIC95]{AngChopIlm}
S.~Angenent, T.~Ilmanen, and D.~L. Chopp.
\newblock A computed example of nonuniqueness of mean curvature flow in {$\bold
  R^3$}.
\newblock {\em Comm. Partial Differential Equations}, 20(11-12):1937--1958,
  1995.

\bibitem[AIV02]{angenent2002fattening}
Sigurd~B Angenent, Tom Ilmanen, and Juan~JL Vel{\'a}zquez.
\newblock Fattening from smooth initial data in mean curvature flow.
\newblock 2002.

\bibitem[AK22]{AK}
Sigurd~B. Angenent and Dan Knopf.
\newblock Ricci solitons, conical singularities, and nonuniqueness.
\newblock {\em Geom. Funct. Anal.}, 32(3):411--489, 2022.

\bibitem[BC23]{BamChen}
Richard Bamler and Eric Chen.
\newblock Degree theory for 4-dimensional asymptotically conical gradient
  expanding solitons.
\newblock {\em \url{https://arxiv.org/abs/2305.03154}}, 2023.

\bibitem[BK23]{BamlerKleiner23}
Richard Bamler and Bruce Kleiner.
\newblock On the multiplicity one conjecture for mean curvature flows of
  surfaces.
\newblock {\em \url{https://arxiv.org/abs/2312.02106}}, 2023.

\bibitem[Bre16]{Brendle:genus0}
Simon Brendle.
\newblock Embedded self-similar shrinkers of genus 0.
\newblock {\em Ann. of Math. (2)}, 183(2):715--728, 2016.

\bibitem[BW21a]{BW:compact}
Jacob Bernstein and Lu~Wang.
\newblock Smooth compactness for spaces of asymptotically conical
  self-expanders of mean curvature flow.
\newblock {\em Int. Math. Res. Not. IMRN}, (12):9016--9044, 2021.

\bibitem[BW21b]{BW:space}
Jacob Bernstein and Lu~Wang.
\newblock The space of asymptotically conical self-expanders of mean curvature
  flow.
\newblock {\em Math. Ann.}, 380(1-2):175--230, 2021.

\bibitem[BW22a]{BW:schoen}
Jacob Bernstein and Lu~Wang.
\newblock Closed hypersurfaces of low entropy in {$\Bbb R^4$} are isotopically
  trivial.
\newblock {\em Duke Math. J.}, 171(7):1531--1558, 2022.

\bibitem[BW22b]{BW:mount}
Jacob Bernstein and Lu~Wang.
\newblock A mountain-pass theorem for asymptotically conical self-expanders.
\newblock {\em Peking Math. J.}, 5(2):213--278, 2022.

\bibitem[BW22c]{BW19}
Jacob Bernstein and Lu~Wang.
\newblock Relative expander entropy in the presence of a two-sided obstacle and
  applications.
\newblock {\em Adv. Math.}, 399:Paper No. 108284, 48, 2022.

\bibitem[BW22d]{BW:top-unique-expand}
Jacob Bernstein and Lu~Wang.
\newblock Topological uniqueness for self-expanders of small entropy.
\newblock {\em Camb. J. Math.}, 10(4):785--833, 2022.

\bibitem[BW23]{BW18}
Jacob Bernstein and Lu~Wang.
\newblock An integer degree for asymptotically conical self-expanders.
\newblock {\em Calc. Var. Partial Differential Equations}, 62(7):Paper No. 200,
  46, 2023.

\bibitem[CCMS24a]{CCMS:generic1}
Otis Chodosh, Kyeongsu Choi, Christos Mantoulidis, and Felix Schulze.
\newblock Mean curvature flow with generic initial data.
\newblock {\em Invent. Math.}, 237(1):121--220, 2024.

\bibitem[CCMS24b]{CCMS:low-ent-gen}
Otis Chodosh, Kyeongsu Choi, Christos Mantoulidis, and Felix Schulze.
\newblock Mean curvature flow with generic low-entropy initial data.
\newblock {\em Duke Math. J.}, 173(7):1269--1290, 2024.

\bibitem[CCS23]{ccs23}
Otis Chodosh, Kyeongsu Choi, and Felix Schulze.
\newblock Mean curvature flow with generic initial data {II}.
\newblock {\em \url{https://arxiv.org/abs/2302.08409}}, 2023.

\bibitem[CGG91]{CGG}
Yun~Gang Chen, Yoshikazu Giga, and Shun'ichi Goto.
\newblock Uniqueness and existence of viscosity solutions of generalized mean
  curvature flow equations.
\newblock {\em J. Differential Geom.}, 33(3):749--786, 1991.

\bibitem[Che22a]{chen:existunique}
Letian Chen.
\newblock On the existence and uniqueness of ancient rescaled mean curvature
  flows.
\newblock {\em \url{https://arxiv.org/abs/2212.10798}}, 2022.

\bibitem[Che22b]{Chen:rot1}
Letian Chen.
\newblock Rotational symmetry of solutions of mean curvature flow coming out of
  a double cone.
\newblock {\em J. Geom. Anal.}, 32(10):Paper No. 250, 11, 2022.

\bibitem[Che23]{Chen:rot2}
Letian Chen.
\newblock Rotational symmetry of solutions of mean curvature flow coming out of
  a double cone {II}.
\newblock {\em Calc. Var. Partial Differential Equations}, 62(2):Paper No. 70,
  32, 2023.

\bibitem[CHH22]{chh18}
Kyeongsu Choi, Robert Haslhofer, and Or~Hershkovits.
\newblock Ancient low-entropy flows, mean-convex neighborhoods, and uniqueness.
\newblock {\em Acta Math.}, 228(2):217--301, 2022.

\bibitem[CHHW22]{ChoiHaslhoferHershkovitsWhite}
Kyeongsu Choi, Robert Haslhofer, Or~Hershkovits, and Brian White.
\newblock Ancient asymptotically cylindrical flows and applications.
\newblock {\em Invent. Math.}, 229(1):139--241, 2022.

\bibitem[CS21]{ChoSch}
Otis Chodosh and Felix Schulze.
\newblock Uniqueness of asymptotically conical tangent flows.
\newblock {\em Duke Math. J.}, 170(16):3601--3657, 2021.

\bibitem[CY07]{CY:pseudo}
Bing-Long Chen and Le~Yin.
\newblock Uniqueness and pseudolocality theorems of the mean curvature flow.
\newblock {\em Comm. Anal. Geom.}, 15(3):435--490, 2007.

\bibitem[Der16]{Der16}
Alix Deruelle.
\newblock Smoothing out positively curved metric cones by {R}icci expanders.
\newblock {\em Geom. Funct. Anal.}, 26(1):188--249, 2016.

\bibitem[DG94]{DG:new}
Ennio De~Giorgi.
\newblock New ideas in calculus of variations and geometric measure theory.
\newblock In {\em Motion by mean curvature and related topics ({T}rento,
  1992)}, pages 63--69. de Gruyter, Berlin, 1994.

\bibitem[DH22]{Daniels-Holgate}
J.~M. Daniels-Holgate.
\newblock Approximation of mean curvature flow with generic singularities by
  smooth flows with surgery.
\newblock {\em Adv. Math.}, 410(part A):Paper No. 108715, 42, 2022.

\bibitem[Din20]{DingExp}
Qi~Ding.
\newblock Minimal cones and self-expanding solutions for mean curvature flows.
\newblock {\em Math. Ann.}, 376(1-2):359--405, 2020.

\bibitem[DS20]{DS20}
Alix Deruelle and Felix Schulze.
\newblock Generic uniqueness of expanders with vanishing relative entropy.
\newblock {\em Math. Ann.}, 377(3-4):1095--1127, 2020.

\bibitem[DS23]{DS:rel}
Alix Deruelle and Felix Schulze.
\newblock A relative entropy and a unique continuation result for {R}icci
  expanders.
\newblock {\em Comm. Pure Appl. Math.}, 76(10):2613--2692, 2023.

\bibitem[ES91]{EvansSpruck1}
L.~C. Evans and J.~Spruck.
\newblock Motion of level sets by mean curvature. {I}.
\newblock {\em J. Differential Geom.}, 33(3):635--681, 1991.

\bibitem[FP96]{NP:numerics}
Francesca Fierro and Maurizio Paolini.
\newblock Numerical evidence of fattening for the mean curvature flow.
\newblock {\em Math. Models Methods Appl. Sci.}, 6(6):793--813, 1996.

\bibitem[GS18]{GS}
Panagiotis Gianniotis and Felix Schulze.
\newblock Ricci flow from spaces with isolated conical singularities.
\newblock {\em Geom. Topol.}, 22(7):3925--3977, 2018.

\bibitem[Hel12]{Helm}
Sebastian Helmensdorfer.
\newblock A model for the behavior of fluid droplets based on mean curvature
  flow.
\newblock {\em SIAM J. Math. Anal.}, 44(3):1359--1371, 2012.

\bibitem[Hui90]{Huisken:sing}
Gerhard Huisken.
\newblock Asymptotic behavior for singularities of the mean curvature flow.
\newblock {\em J. Differential Geom.}, 31(1):285--299, 1990.

\bibitem[HW20]{hershwhite}
Or~{Hershkovits} and Brian {White}.
\newblock {Nonfattening of mean curvature flow at singularities of mean convex
  type}.
\newblock {\em {Commun. Pure Appl. Math.}}, 73(3):558--580, 2020.

\bibitem[{Ilm}94]{ilmanen}
Tom {Ilmanen}.
\newblock {\em {Elliptic regularization and partial regularity for motion by
  mean curvature}}, volume 520.
\newblock Providence, RI: American Mathematical Society (AMS), 1994.

\bibitem[{Ilm}95a]{Trieste}
Tom {Ilmanen}.
\newblock {Lectures on mean curvature flow and related equations ({T}rieste
  notes)}, 1995.

\bibitem[Ilm95b]{Ilmanen:singularities}
Tom Ilmanen.
\newblock Singularities of mean curvature flow of surfaces.
\newblock {\em \url{https://people.math.ethz.ch/~ilmanen/papers/sing.ps}},
  1995.

\bibitem[Ilm03]{ilmanen2003problems}
Tom Ilmanen.
\newblock Problems in mean curvature flow.
\newblock {\em
  \url{https://people.math.ethz.ch/~ilmanen/classes/eil03/problems03.pdf}},
  2003.

\bibitem[INS19]{INS}
Tom {Ilmanen}, Andr\'e {Neves}, and Felix {Schulze}.
\newblock {On short time existence for the planar network flow}.
\newblock {\em {J. Differ. Geom.}}, 111(1):39--89, 2019.

\bibitem[MY82]{MeeksYau}
William~W. Meeks, III and Shing~Tung Yau.
\newblock The existence of embedded minimal surfaces and the problem of
  uniqueness.
\newblock {\em Math. Z.}, 179(2):151--168, 1982.

\bibitem[OS88]{OS}
Stanley Osher and James~A. Sethian.
\newblock Fronts propagating with curvature-dependent speed: algorithms based
  on {H}amilton-{J}acobi formulations.
\newblock {\em J. Comput. Phys.}, 79(1):12--49, 1988.

\bibitem[SS93]{SonSou}
H.~M. Soner and P.~E. Souganidis.
\newblock Singularities and uniqueness of cylindrically symmetric surfaces
  moving by mean curvature.
\newblock {\em Comm. Partial Differential Equations}, 18(5-6):859--894, 1993.

\bibitem[SS13]{ShulzeSimon}
Felix Schulze and Miles Simon.
\newblock Expanding solitons with non-negative curvature operator coming out of
  cones.
\newblock {\em Math. Z.}, 275(1-2):625--639, 2013.

\bibitem[SW20]{SchulzeWhite}
Felix Schulze and Brian White.
\newblock A local regularity theorem for mean curvature flow with triple edges.
\newblock {\em J. Reine Angew. Math.}, 758:281--305, 2020.

\bibitem[Whi94]{White:questDG}
Brian White.
\newblock Some questions of {D}e {G}iorgi about mean curvature flow of
  triply-periodic surfaces.
\newblock In {\em Motion by mean curvature and related topics ({T}rento,
  1992)}, pages 210--213. de Gruyter, Berlin, 1994.

\bibitem[Whi95]{White:topology-weak}
Brian White.
\newblock The topology of hypersurfaces moving by mean curvature.
\newblock {\em Comm. Anal. Geom.}, 3(1-2):317--333, 1995.

\bibitem[Whi02]{White:ICM}
Brian White.
\newblock Evolution of curves and surfaces by mean curvature.
\newblock In {\em Proceedings of the {I}nternational {C}ongress of
  {M}athematicians, {V}ol. {I} ({B}eijing, 2002)}, pages 525--538. Higher Ed.
  Press, Beijing, 2002.

\bibitem[Whi05]{White:Brakke}
Brian White.
\newblock A local regularity theorem for mean curvature flow.
\newblock {\em Ann. of Math. (2)}, 161(3):1487--1519, 2005.

\end{thebibliography}
\end{document}